\documentclass{amsart}
\usepackage[frenchb]{babel}
\usepackage[utf8]{inputenc}
\usepackage[T1]{fontenc}
\usepackage{textcomp}
\usepackage{amsfonts}
\usepackage{amssymb}
\usepackage{graphicx}
\let\oldref\ref
\usepackage{hyperref}
\hypersetup{
    colorlinks = true,
}
\newcommand{\B}{{\mathbb B}}
\newcommand{\C}{{\mathbb C}}
\newcommand{\R}{{\mathbb R}}
\newcommand{\N}{{\mathbb N}}
\renewcommand{\S}{{\mathbb S}}
\newcommand{\Z}{{\mathbb Z}}
\newcommand{\ZnZ}{{\Z/N.\Z}}
\newcommand{\RnZ}{{\R/N.\Z}}

\renewcommand{\AA}{{\mathcal A}}
\newcommand{\BB}{{\mathcal B}}
\newcommand{\CC}{{\mathcal C}}
\newcommand{\HH}{{\mathcal H}}
\newcommand{\LL}{{\mathcal L}}
\newcommand{\MM}{{\mathcal M}}
\newcommand{\PP}{{\mathcal P}}
\renewcommand{\SS}{{\mathcal S}}
\newcommand{\TT}{{\mathcal T}}

\renewcommand{\Re}{\text{Re}}
\renewcommand{\Im}{\text{Im}}

\newcommand{\demi}{\text{\textonehalf}}
\newcommand{\vect}[2][]{\mathop{\mathrm{Vect}}_{#1}\{#2\}}
\newcommand{\kernel}[1]{\mathop{\mathrm{Ker}}\left(#1\right)}

\newcommand{\pdtscalaire}[2]{\left(#1\mathbin{|}#2\right)}
\newcommand{\pdthermitien}[2]{\left<#1\mathbin{|}#2\right>}
\newcommand{\opdiag}[1]{\left[\!\left[#1\right]\!\right]}
\newcommand{\carac}{{1\!\!1}}

\newcommand{\parmi}[2]{
  \begin{pmatrix}
    #2 \\ #1
  \end{pmatrix}
}

\newcommand{\signe}{\mathop{\mathrm sgn}}

\newtheorem{lemma}{Lemme}
\newtheorem{thm}{Théorème}
\newtheorem{proposition}{Proposition}
\newtheorem{definition}{Définition}

\newcommand{\ancetre}[1]{{\sc #1}}
\newcommand{\ie}{{\it i.e. }}
\newcommand{\cf}{{\it c.f. }}
\newcommand{\ssi}{\ensuremath{\Leftrightarrow}}
\newcommand{\et}{\ensuremath{\text{ et }}}

\begin{document}
\title{Aire, périmètre et polygones cocycliques}
\author{JC Leger}
\email{jcleger75@gmail.com}%
\date{14 mai 2018}
\maketitle
\tableofcontents
\clearpage
\section{Introduction}

Dans son traité de géométrie, \cite{LEGENDRE14}, Proposition VII, p.134,
\ancetre{Legendre} popularise le fameux théorème isopérimétrique sur les
polygones plans : 

\begin{thm}\label{thm:Legendre}
Soit $N\geq 4$ un entier naturel et $\ell$ un nombre réel  positif. Parmi les polygones à $N$ cotés de périmètre $\ell$, le polygone d'aire maximale est le polygone régulier à $N$ cotés de périmètre $\ell$.
\end{thm}

Ce résultat est probablement du à \ancetre{L'Huilier}-- il est présent
(p.107, paragraphe 13) dans 
\cite{LHUILIER89}, et semble dater de 1782-- dont la démonstration est reprise {\it verbatim} par
\ancetre{Legendre}. 
Cette démonstration est basée sur un résultat plus méconnu
\footnote{\cite{LHUILIER89}, p.107, paragraphe 13, \cite{LEGENDRE14},
Proposition VI, p.133, il est à noter que ce théorème disparait dès
l'édition suivante, \cite{LEGENDRE15}, alors que \ancetre{Legendre} est
mort en 1833} :


\begin{thm}\label{thm:Huilier}
Soit $N\geq 4$ un entier naturel et un $N$-uplet $\ell_1,\dots,\ell_N$ de réels strictement positifs. Parmi les polygones à $N$ cotés de longueurs respectives $\ell_1,\dots,\ell_N$, un polygone d'aire maximale est \emph{cocyclique}, \ie inscrit dans un cercle. 
\end{thm}

L'objet du présent article est de présenter une démonstration d'un
résultat\footnote{L'énoncé précis est le Théorème
  \ref{thm:pt-critiques-aire}, il s'agit d'une démonstration
  inédite.} dans cette lignée, comparativement beaucoup plus récent,  du à
\ancetre{Khimshiashvili} et \ancetre{Panina}, \cite{KP2008}, que l'on peut énoncer de façon informelle~:

{\it Sur l'espace des polygones à $N$ cotés dont les longueurs sont imposées,
 les points critiques de l'aire orientée sont les polygones cocyliques.
}

Avant de démontrer cet énoncé, une étape fondamentale est de préciser son sens exact. Cela est fait dans le paragraphe \ref{sec:defs}.
La preuve sera ensuite développée dans les paragraphes \ref{sec:al}, \ref{sec:rep-param-Tz} et \ref{sec:preuve-thm}, le paragraphe \ref{sec:preuves-lemmes} détaillant quelques démonstrations de lemmes techniques dont le développement immédiat aurait nui à la fluidité du discours. Le paragraphe \ref{sec:calc-diff-rappels} apporte, quant à lui, quelques précisions concernant les éléments de géométrie différentielle nécessaires aux calculs.

Nous donnerons aussi des variantes liées à d'autres contraintes dans
le paragraphe \ref{sec:lagrange}, quelques illustrations numériques dans le paragraphe \ref{sec:montee-gradient}
 et concluerons,
dans le paragraphe \ref{sec:applications}, par des
ouvertures possibles concernant les liens entre ces calculs, le
dénombrement de configurations cocycliques, la
recherche effective de
polynômes de \ancetre{Heron}--\ancetre{Robbins}, liant aire et longueur des côtés d'un
polygone cocyclique, et la topologie de certains espaces des polygones.

\section{Définitions, notations et énoncé}\label{sec:defs}

Fixons un entier $N\geq 4$. Nous identifions une
fois pour toute $\C$, le corps des nombres complexes, avec le plan
euclidien orienté, de la manière usuelle. 

Pour un nombre complexe $z\in\C$, on note $\Re(z)$, resp $\Im(z)$, sa partie réelle, resp. sa partie imaginaire.

On note $\S^1$ le cercle
unité du plan,\ie le groupe des nombres complexes de module $1$.

Ce groupe  agit sur $\C$ par multiplication. Cette action est
celle du groupe $SO_2$ des rotations vectorielles du plan euclidien orienté. 

\subsection{Indexation}

Si $A_0,\dots,A_{N-1}$ sont $N$ points du plan, le polygone planaire orienté  $(A_0,\dots,A_{N-1})$ est la suite de ces points, placés dans cet ordre de sorte que l'on puisse considérer les \emph{côtés} de ce polygone, $A_0\to A_1, A_1\to A_2,\dots,A_{N-2}\to A_{N-1}$, sans oublier $A_{N-1}\to A_0$. 

L'indexation des points d'un polygone orienté possède donc intrinsèquement un caractère cyclique. L'indexation de ses côtés aussi. 

Pour tout ce qui concerne les développements théoriques, nous adoptons donc une indexation abstraite en considérant $(S,\Sigma)$ le graphe cyclique orienté sur $N$ sommets.

Du point de vue de l'implémentation des codes de calcul dans le paragraphe \ref{sec:numerics}, on peut choisir une réalisation très concrète de l'ensemble des sommets de ce graphe
$$
S=\{0,\dots,N-1\}
$$

Le groupe $\ZnZ$ agit naturellement sur le graphe $(S,\Sigma)$, l'action de $1$ se traduisant par le décalage de chaque sommet vers le sommet suivant $n\mapsto n+1$, avec comme convention, pour la réalisation concrète
$$
0+1=1,1+1=2\dots,(N-2)+1=(N-1), (N-1)+1=0
$$

Si $n\in S$, $n-1$ est donc le sommet précédant $n$ dans $S$.

Convenons de noter, pour un sommet $n$ de $S$, $n+\demi$ l'arète suivant le sommet $n$ et $n-\demi$, l'arète précédant $n$,\ie
$$ 
n+\demi:=(n,n+1)
\et
n-\demi:=(n-1,n)
$$
Pour la réalisation concrète de $S$, l'ensemble des arêtes se réécrit donc
$$\Sigma=\{0+\demi,1+\demi,\dots, (N-1)+\demi\}=\{n+\demi,n\in S\}$$

Maintenant, pour une arête $\nu\in\Sigma$, $\nu=(n,n+1)=n+\demi$, convenons de noter\footnote{de sorte que $\nu=(\nu-\demi,\nu+\demi)$} $\nu-\demi$ l'extrémité de départ $n$ de $\nu$, $\nu+\demi$ l'extrémité d'arrivée $n+1$ de $\nu$ et $\nu+1=(n+1,n+2)$ l'arète suivant $\nu$, on définit ainsi une action de $\ZnZ$ sur $\Sigma$ et on a, par simple jeu d'écritures,
$$
\forall n\in S,\, 
(n-\demi)-\demi=n-1,\,
(n-\demi)+\demi=n=(n+\demi)-\demi,\,
(n+\demi)+\demi=n+1
$$
\et
$$
\forall \nu\in \Sigma,\, 
(\nu-\demi)-\demi=\nu-1,\,
(\nu-\demi)+\demi=\nu=(\nu+\demi)-\demi,\,
(\nu+\demi)+\demi=\nu+1
$$



Si on utilise les règles de calcul dans $\R$ modulo $N.\Z$, on peut naturellement aussi identifier  $S$, resp. $\Sigma$, avec  $\ZnZ$,
resp. $\ZnZ+\frac12 \subset \RnZ$, ce qui permet de d'identifier $n+\frac12$ avec $n+\demi$, $\nu-\frac12$ avec $\nu-\demi$, etc...

Du point de vue de la notation, les éléments de $S$ et de $\C^S$
seront désignés, autant que possible, par des lettres latines, ceux de $\Sigma$ et de $\C^\Sigma$ par
des lettres grecques.

\begin{definition}
Un $N$-gone orienté planaire est une suite de $N$ nombres complexes indéxée
par $S$. C'est un $S$-uplet, un élément de $\C^S$.  
\end{definition}

Nous noterons
de façon générique un tel polygone $z=(z_n)_{n\in S}$. Le $\Sigma$-uplet des
cotés de ce polygone est la suite de $N$  nombres complexes indexée par
$\Sigma$, l'ensemble des arêtes du graphe, $(\zeta_\nu)_{\nu\in \Sigma}$ avec
$\zeta_\nu=z_{\nu{+\demi}}-z_{\nu{-\demi}}$.

\begin{figure}[h!]
  \centering
  \includegraphics[scale=0.7]{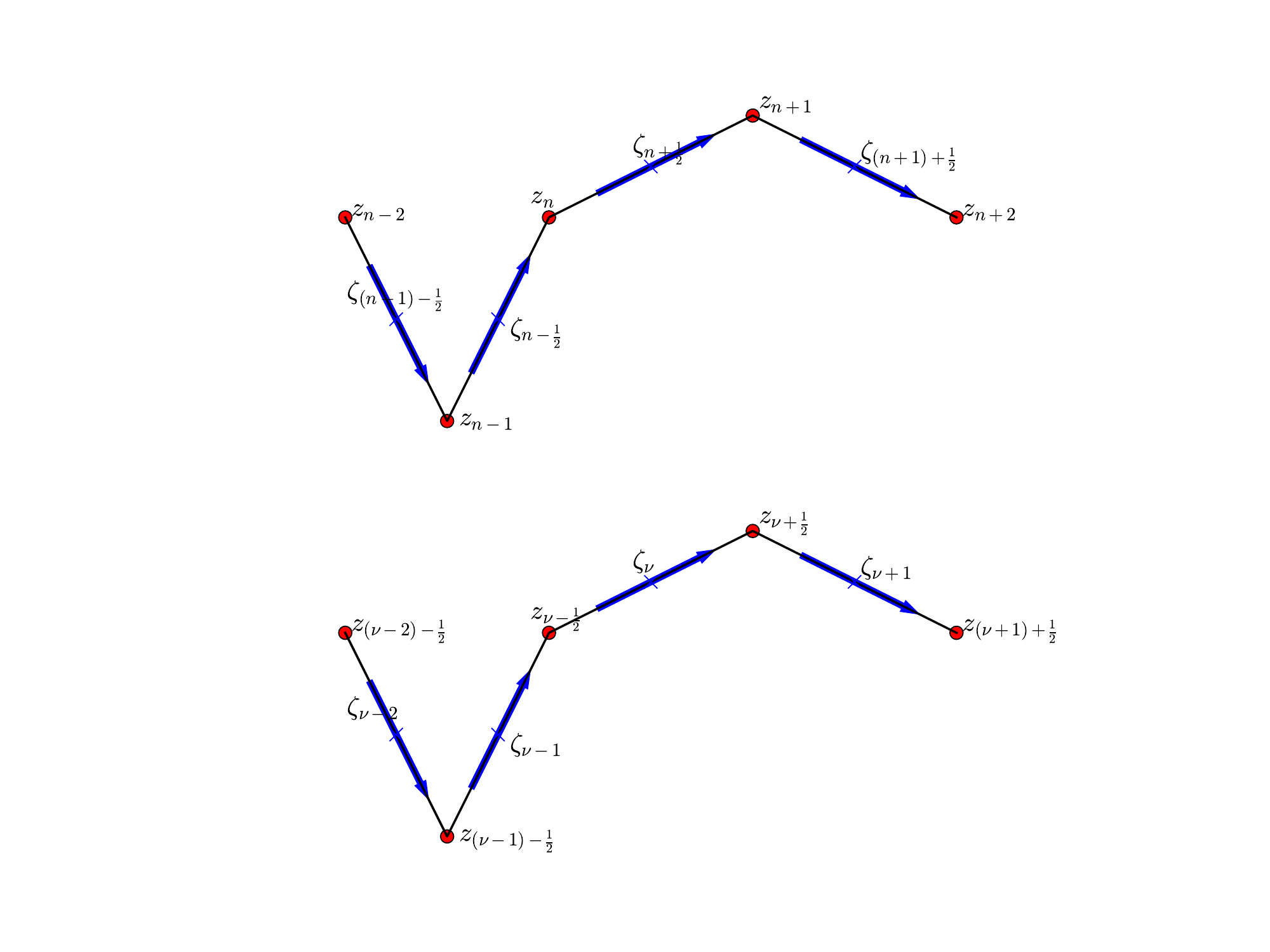}
  \caption{Indexations}
  \label{fig:indexation}
\end{figure}

On note $D$ l'application linéaire qui, à un polygone, associe le
$\Sigma$-uplet de ses
cotés.
\begin{equation}
  \label{eq:def-D}
\begin{array}{cccl}
D:&\C^S &\to& \C^\Sigma\\
  &z=(z_n)_{n\in S}&\mapsto  & \zeta=Dz=(z_{\nu{+\demi}}-z_{\nu{-\demi}})_{\nu\in \Sigma}
\end{array}
\end{equation}

Bien que l'application $D$ ne soit pas bijective, on peut y remédier en posant 
$$\C_0^S=\{z\in \C^S,\, \sum_n z_n=0\}
\et
\C_0^\Sigma=\{\zeta\in \C^\Sigma,\, \sum_\nu \zeta_\nu=0\}
\mathpunct{,}
$$ 
si bien que, par restriction, $D:\C_0^S\to \C_0^\Sigma$ est un isomorphisme.

\subsection{Variétés de polygones}

\subsubsection{Des équations cartésiennes...}

Etant donné un $\Sigma$-uplet, $\ell=(\ell_\nu)_{\nu\in \Sigma}$,
de nombres réels positifs,  on considère l'ensemble $\PP(\ell)$ des
polygones dont les cotés ont pour longueurs $\ell$, \ie
\begin{eqnarray}
  \label{eq:def-Pl}
\PP(\ell)&=&
\{z\in\C^S,\,\forall \nu\in \Sigma,\,|z_{\nu{+\demi}}-z_{\nu{-\demi}}|=\ell_\nu\}\\
&=&
\{z\in\C^S,\,\forall \nu\in \Sigma,\,|(Dz)_\nu|=\ell_\nu\}\nonumber
\end{eqnarray}

Il est bien connu (c'est un exercice\footnote{Un indice: prendre
  $\nu_0$ tel que $\ell_{\nu_0}$ est maximal. Considérer
  $p\in\{1,..,N-1\}$ minimal pour la propriété $\sum_{q=0}^p
  \ell_{\nu_0+q}> \frac12\sum_{\nu}\ell_\nu$, vérifier que $p\leq N-2$ et considérer le triangle
  de cotés $\sum_{q=0}^{p-1}  \ell_{\nu_0+q}$, $\ell_{\nu_0+p}$ et $\sum_{q=p+1}^{N-1}  \ell_{\nu_0+q}$}  basé sur l'inégalité
triangulaire) que $\PP(\ell)$ est non vide si et seulement si
\begin{equation}
  \label{eq:PP-ND1}
  \forall \nu\in \Sigma,\, \ell_\nu \leq \frac12 \sum_{\nu'\in \Sigma} \ell_{\nu'}.
\end{equation}
Il est par ailleurs bien connu que si l'une des inégalités \eqref{eq:PP-ND1}
est une égalité alors l'indice $\nu_0$ pour lequel ceci se produit est
unique et les polygones dans $\PP(\ell)$ sont tous dégénérés, \ie, les
points d'un tel polygone sont situés sur le côté $\nu_0$, de longueur
$\ell_{\nu_0}$. Nous excluons ces configurations triviales de nos considérations et
supposons donc dorénavant que
\begin{equation}
  \label{eq:PP-ND1-strict}
  \forall \nu\in \Sigma,\, 0<\ell_\nu < \frac12 \sum_{\mu\in \Sigma} \ell_\mu.
\end{equation}

Un tel vecteur $\ell$ sera appelé un \emph{vecteur de longueurs admissible}.

Evidemment, ce qui nous intéresse vraiment, ce sont les polygones à isométrie
(directe) près.

Une première normalisation pratique, qui préserve les symétries naturelles
entre les indices, consiste à travailler sur l'ensemble des \emph{polygones centrés}
\begin{equation}
  \label{eq:def-Pl0}
\PP_0(\ell)=\PP(\ell)\cap\C_0^S=\{z\in\PP(\ell),\,\sum_{n\in S}z_n =0\}  
\end{equation}

Comme nous voulons faire du calcul différentiel et étudier les points critiques de la fonction à valeurs réelles \og aire orientée \fg{} sur l'\og espace des polygones à longueurs prescrites\fg\ , l'idéal est de disposer d'une \emph{sous-variété différentielle} d'un espace vectoriel réel de dimension finie.

Cela sera expliqué plus loin mais nous constatons pour le moment que par définition, les ensembles $\PP(\ell)$ et $\PP_0(\ell)$ sont des \textit{sous-variétés algébriques réelles} d'un espace vectoriel réel de dimension finie, $\C^S\simeq (\R^2)^S\simeq \R^{2N}$,\ie le lieu des zéros communs d'une famille de polynômes réels à $\dim_{\R}\C^S=2N$ indéterminées.

Pour $\nu \in \Sigma$, $z\in \C^S$, posons
\begin{equation}
  \label{eq:def-Lnu}
L^2_\nu(z)=|z_{\nu{+\demi}}-z_{\nu{-\demi}}|^2
\mathpunct{.}  
\end{equation}

En exprimant $L^2_\nu$ en fonction des parties réelles et imaginaires de $z_{\nu{+\demi}}$ et $z_{\nu{-\demi}}$, il apparait clairement que $L^2_\nu$ est  un polynome quadratique réel sur $\C^S\simeq (\R^2)^S$.

On a alors, comme annoncé,  que $\PP(\ell)$ et $\PP_0(\ell)$ sont les lieux des zéros communs d'une famille de polynômes réels~:
\begin{eqnarray}\nonumber 
\PP(\ell)& =& \{z\in\C^S,\, \forall\nu\in \Sigma,\,L^2_\nu(z)=\ell_\nu^2\} ,\\ \nonumber
\PP_0(\ell)&=& \{z\in\C^S,\,\forall\nu\in \Sigma,\, L^2_\nu(z)=\ell_\nu^2\}\cap \C^S_0
.
\end{eqnarray}

Remarquons qu'on utilisé le fait que $\C^S_0$ est lui même une sous-variété de $\C^S$ avec 
$$
\C_0^S = \{z\in\C^S,\, \Re(\sum_n z_n)=0\}\cap \{z\in\C^S,\, \Im(\sum_n z_n)=0\}
$$
En $z\in\C_0^S$, l'espace tangent à $\C_0^S$ est $\C_0^S$.



Concernant le problème d'extremum qui nous intéresse, l'avantage majeur de $\PP_0(\ell)$ par rapport à $\PP(\ell)$ est sa compacité.

\begin{proposition}\label{compa}
Soit $\ell$ un vecteur de longueurs admissible. 
L'ensemble $\PP_0(\ell)$ est une partie compacte de $\C^S$.
\end{proposition}
\begin{proof}
Il est clair que $\PP_0(\ell)$ est une partie fermée, non vide en raison de la continuité des applications $L_\nu^2-\ell_\nu^2$ et du caractère admissible du vecteur $\ell$.

La compacité de $\PP_0(\ell)$ est donc conséquence de son caractère borné et du fait que $\C^S$ est de dimension finie.
 
 Comme les longueurs des arêtes des polygones de $\PP_0(\ell)$ sont fixées, la distance entre deux sommets quelconques d'un tel polygone est majorée par le périmètre du polygone, qui vaut $P(\ell)=\sum_\nu \ell_\nu$.

Un polygone $z\in\PP_0(\ell)$ est donc contenu dans un disque plan (par exemple centré en l'un des sommets) de rayon $P(\ell)$.

Par convexité, ce disque contient le centre de gravité de $z$, qui, et c'est le point de la définition de $\PP_0(\ell)$, est $0$.

Le polygone $z$ est donc contenu dans le disque plan de centre $0$ et de rayon $2P(\ell)$, et ceci montre que, dans $\C^S$, la partie $\PP_0(\ell)$ est contenue dans la boule (euclidienne) de centre $0$, de rayon $2P(\ell)$.

\end{proof}

\subsubsection{...aux sous-variétés}

Nous donnons au paragraphe \ref{sec:calc-diff-rappels} quelques précisions concernant le calcul différentiel réel dans $\C^S$ et les notions de sous-variétés algébriques et différentielles de ce $\R$-espace vectoriel.

Dans le cas de $\PP_0(\ell)$, cette sous-variété algébrique réelle de $\C^S$ admet en tout point $z$  un espace tangent $\TT_z$ au sens \emph{réel algébrique}  défini
comme l'intersection des noyaux des différentielles en $z$ des
polynômes définissants.

En l'occurence, comme on a
$$\PP_0(\ell)=\{z\in\C^S,\,\forall\nu\in\Sigma,\, 
L^2_\nu(z)=\ell_\nu^2
\}\cap \C^S_0
\mathpunct{,}
$$
alors
l'espace tangent $\TT_z$ à $\PP_0(\ell)$ en $z\in\PP_0(\ell)$ est
défini par 
\begin{equation}
  \label{eq:cartesienne-Tz-00}
  \TT_z = \{t=(t_n)\in\C^S,\, \forall\nu\in\Sigma,\, 
(d_z L^2_\nu).t=0
\}\cap \C^S_0
\end{equation}
que nous expliciterons en l'équation \eqref{eq:cartesienne-Tz}.

Remarquons que pour faire ce calcul, on utilisé le fait que $\C^S_0$ est lui même une sous-variété de $\C^S$ avec 
$$
\C_0^S = \{z\in\C^S,\, \Re(\sum_n z_n)=0\}\cap \{z\in\C^S,\, \Im(\sum_n z_n)=0\}
$$

La dimension réelle de $\TT_z$, intersection de $N+2$ noyaux de formes linéaires dans un espace de dimension $2N$ est supérieure à $N-2$. Si celle-ci vaut $N-2$, \ie si les formes linéaires sont indépendantes, on dit que le point $z$ est un point \emph{régulier} de $\PP_0(\ell)$, sinon, on dit que le point $z$ est un point \emph{singulier} de $\PP_0(\ell)$.

La proposition \ref{prop:Pl-variete} identifie les points singuliers de $\PP(\ell)$ et de $\PP_0(\ell)$ et donne la condition  nécessaire et suffisante sur $\ell$ pour que  $\PP(\ell)$ et $\PP_0(\ell)$ soient des sous-variétés différentielles de $\C^S$, \ie n'aient pas de points singuliers. Cette condition sur $\ell$ montre que, pour $\ell$ admissible et générique\footnote{en tous sens communément acceptés de ce mot}, $\PP_0$ et $\PP_\ell$ sont des sous-variétés différentielles de $\C^S$.

Les calculs faits donnent au passage, \cf Eq. \eqref{eq:parametrage-Tz}, une
description paramétrique de l'espace tangent 
$\TT_z$.

\subsubsection{Un quotient.}
L'{\it espace des  modules} $\MM(\ell)$ associé à $\ell$ est l'ensemble $\PP(\ell)$
modulo l'action naturelle des isométries directes de $\C\simeq \R^2$.

L'espace des modules $\MM(\ell)$ associé à $\ell$ s'identifie donc
naturellement à $\PP_0(\ell)$ modulo l'action diagonale sur $\C^S$ du cercle
$\S^1$, \ie l'action de $SO_2$ sur les polygones.

$D$ induit une bijection de $\PP_0(\ell)$ sur 
$$
\{(\zeta_\nu)_{\nu\in\Sigma},\,\forall
\nu\in\Sigma,\,|\zeta_\nu|=\ell_\nu \text{ et } \sum_{\nu\in\Sigma}\zeta_\nu=0\}
$$
lui-même en bijection naturelle, par normalisation, avec 
$$
\{(u_\nu)_{\nu\in\Sigma}\in(\S^1)^\Sigma,\,\sum_{\nu\in\Sigma}\ell_\nu.u_\nu=0\}
$$

L'espace $\MM(\ell)$ est étudié par exemple dans \cite{FS07} où il est
défini comme étant 
$$
\{(u_\nu)_{\nu\in\Sigma}\in(\S^1)^\Sigma,\,\sum_{\nu\in\Sigma}\ell_\nu.u_\nu=0\}/\S^1
$$

Cette présentation de $\MM(\ell)$ a l'avantage d'exhiber clairement
la correspondance naturelle entre les espaces $\MM(\ell)$ et $\MM(\sigma.\ell)$ où, pour
une permutation $\sigma$ de $\Sigma$, on note $\sigma.\ell=(\ell_{\sigma(\nu)})_{\nu\in\Sigma}$.

Génériquement en $\ell$, la sous-variété  $\PP(\ell)$ a pour dimension (réelle) $2N-N=N$, la sous-variété compacte $\PP_0(\ell)$ a pour dimension $N-2$ et la variété\footnote{elle n'est pas, en tant que telle, une sous-variété de $\C^S$, en conséquence, pour garder un niveau d'abstraction raisonnable, nous éviterons de calculer directement sur $\MM(\ell)$. }
compacte $\MM(\ell)$ a pour dimension $N-3$. 

Enfin, pour conclure cette section, on remarque qu'on choisit
volontairement de ne pas identifer deux polygones symétriques par rapport à une droite. 

\subsection{L'aire orientée}

Si $\Delta_0=(0,z_0,z_1)$ est un triangle plan orienté positivement,\ie si $(z_0,z_1)$ est une base directe de $\R^2$, $z_0=x_0+iy_0$, $z_1=x_1+iy_1$, son aire, au sens le plus naïf du terme se calcule avec le déterminant donnant l'aire du parallèlogramme $(0,z_0,z_0+z_1,z_1)$ 
$$
A(\Delta_0)=\frac12
\begin{vmatrix}
x_0 & x_1\\
y_0 & y_1
\end{vmatrix}=\frac12(x_0.y_1-x_1.y_0) 
=\frac1{2i}\Im(\overline{z_0}.z_1) = \frac1{4i}(\overline{z_0}.z_1-z_0.\overline{z_1}) 
$$
Si maintenant $z=(z_{k})_{k\in S}$ est un polygone convexe, contenant $0$ en son intérieur, orienté positivement, \ie, si chaque triangle $\Delta_k=(0,z_k,z_{k+1})$ est orienté positivement, son aire est la somme des aires des triangles $\Delta_k$, \ie
$$
A(z)= \sum_{k\in S}\frac1{4i}(\overline{z_k}.z_{k+1}-z_{k}.\overline{z_{k+1}}) 
$$
En usant du caractère cyclique de $S$, on a donc aussi
\begin{eqnarray*}
A(z)&=& \sum_{k\in S}\frac1{4i}(\overline{z_k}.z_{k+1}-z_{k-1}.\overline{z_{k}}) \\ &=&  \frac1{4i}\sum_{k\in S}\overline{z_k}.(z_{k+1}-z_{k-1})   
= \frac12\sum_k \Im((z_{k+1}-z_{k})\overline{z}_k)
\end{eqnarray*}

On remarque que cette quantité, $A(z)$, définie pour tout polygone $z$, est invariante par translations, rotations directes du plan et que $A(z)=-A(\overline{z})$ où $\overline{z}$ est le symétrique orthogonal du polygone $z$ par rapport à l'axe des abscisses.

On appelle $A(z)$ l'\emph{aire algébrique} du polygone $z$.
La valeur absolue de cette aire coincide avec l'aire naïve au cas où le polygone $z$ est \emph{simple} et donc admet une traingulation.

On définit donc la fonction \og aire algébrique\fg\, $A:\C^S\to \R$~: à un polygone $z=(z_n)\in\C^S$, elle associe
\begin{equation}
  \label{eq:def-aire}
  A(z)=\frac1{4i}\sum_k (z_{k+1}-z_{k-1})\overline{z}_k=\frac12\sum_k \Im((z_{k+1}-z_k)\overline{z}_k)
\end{equation}

La fonction $A$ est invariante sous l'action des translations et sous
l'action diagonale de $\S^1$ sur $\C^S$ : 
\begin{equation}
  \label{eq:action-S1-aire}
\forall \theta\in\R,\,A(e^{i\theta}.z) = A(z).  
\end{equation}
Elle définit donc une fonction sur chaque espace $\MM(\ell)$.

$A$ est clairement un polynome
quadratique réel sur $\C^S\simeq (\R^2)^S$. Calculons sa différentielle en $z\in\C^S$ en utilisant la règle de différentiation des fonctions polynomiales en $z_k$, $\overline{z}_k$ signalée en fin du paragraphe \ref{sec:calc-diff-rappels}. On a par exemple,  pour $k\in S$,
\begin{eqnarray*}
d(z_{k+1}.\overline{z}_k)
&=&\frac{\partial (z_{k+1}.\overline{z}_k)}{\partial z_{k+1}}.dz_{k+1}
+\frac{(\partial z_{k+1}.\overline{z}_k)}{\partial \overline{z}_k}.d\overline{z}_k 
=\overline{z}_k.dz_{k+1}+z_{k+1}.d\overline{z}_k\\
d(z_{k-1}.\overline{z}_k)
&=&\overline{z}_k.dz_{k-1}+z_{k-1}.d\overline{z}_k 
\end{eqnarray*}
et donc, en sommant et en décalant les indices dans les deux dernières sommes,
\begin{eqnarray*}
  d_zA&=& \frac1{4i}\sum_k (z_{k+1}-z_{k-1}).d\overline{z}_k
+\frac1{4i}\sum_k \overline{z}_k.dz_{k+1}
-\frac1{4i}\sum_k \overline{z}_k.dz_{k-1}\\
&=&\frac1{4i}\sum_k (z_{k+1}-z_{k-1}).d\overline{z}_k
+\frac1{4i}\sum_k (\overline{z}_{k-1}-\overline{z}_{k+1}).dz_{k}\\
&=&\frac12 
\sum_k \Im((z_{k+1}-z_{k-1}).d\overline{z}_k)
\end{eqnarray*}

 La différentielle de $A$ en $z\in\C^S$ vaut
\begin{equation}
  \label{eq:diff-A}
d_z A = \frac12\sum_k \Im((z_{k+1}-z_{k-1})d\overline{z}_k)  
\end{equation}
C'est la forme $\R$-linéaire sur le $\R$-espace vectoriel $\C^S$ définie par
$$
\forall t\in\C^S,\,
d_z A.t = \frac12\sum_k \Im((z_{k+1}-z_{k-1}).\overline{t}_k)  
$$

La fonction $A$, continue, admet un minimum et un maximum sur chaque
espace compact $\PP_0(\ell)$. Comme lorsque deux polygones $z$ et $z'$ sont symétriques relativement à une droite, on a $A(z)=-A(z')$, on en déduit que
maximum et minimum sur $\PP_0(\ell)$ sont opposés. Par le principe de \ancetre{Fermat}, en un extremant $z$, régulier pour $\PP_0(\ell)$, $d_z A$ s'annule sur $\TT_z$.
On pose la définition
\begin{definition}
Soit $z=(z_n)\in \PP_0(\ell)$. On dit que $z$ est
    {\it critique} pour $A$ sur $\PP_0(\ell)$ si 
    \begin{enumerate}
    \item $z$ est un point régulier de $\PP_0(\ell)$ 
    \item et $d_z A$ est nulle sur $\TT_z$.
    \end{enumerate}  
\end{definition}
Le résultat principal s'énonce alors
\begin{thm}\label{thm:pt-critiques-aire}
Un polygone $z=(z_n)\in\PP_0(\ell)$ est critique pour $A$ sur
$\PP_0(\ell)$  si et seulement si les points $z_n$ sont cocycliques,
non  alignés. 
\end{thm}

\section{Un peu d'algèbre linéaire et bilinéaire}\label{sec:al}

Un point dans ce qui va être développé ensuite peut mener à confusion
si l'on n'y prend pas garde. 

Pour utiliser les facilités de calculs
qu'offre le corps $\C$ en géométrie plane réelle, nous considérons les
ensembles $\C^S$ et $\C^\Sigma$ qui sont naturellement des $\C$-espaces vectoriels.

Le problème est cependant essentiellement un problème {\it réel}. $\PP_0(\ell)$ est une sous-variété réelle de
$\C^S$, son espace tangent en chaque chaque point est un sous-espace
vectoriel du $\R$-espace vectoriel $\C^S$ et la fonction $A$ est à valeurs réelles.

On verra lors de l'établissement de l'équation
\eqref{eq:parametrage-Tz}, que l'on va paramétrer (non injectivement) $\TT_z$, le sous-espace
tangent à $\PP_0(\ell)$ en $z$, par $\SS_z$, un $\C$-sev de $\C^S$, via une
application $\R$-linéaire. Le calcul de la dimension (réelle) de $\TT_z$, qui nous intéresse en premier lieu, se fera via le calcul de la dimension (complexe) de $\SS_z$.


\subsection{Produit scalaire, produit hermitien}

On munit $\C^S$ et $\C^\Sigma$ de leur structure
hermitienne usuelle, si $z,w\in \C^S$, $\zeta,\omega\in\C^\Sigma$,
\begin{equation}
  \label{eq:def-pdt-hermitien}
\pdthermitien{z}{w}:=\sum_{n\in S}z_n\overline{w}_n,\,\pdthermitien{\zeta}{\omega}:=\sum_{\nu\in \Sigma}\zeta_\nu\overline{\omega}_\nu\mathpunct{.}  
\end{equation}

Ces structures hermitiennes induisent, à l'habitude, des structures
euclidiennes sur les $\R$-espaces vectoriels $\C^S$ et $\C^\Sigma$ via la
définition des produits scalaires, si $z,w\in \C^S$, $\zeta,\omega\in\C^\Sigma$,

\begin{equation}
  \label{eq:def-pdt-scalaire}
\pdtscalaire{z}{w}:=\Re\pdthermitien{z}{w}=\Re\sum_{n\in S}z_n\overline{w}_n,\,\pdtscalaire{\zeta}{\omega}:=\Re\pdthermitien{\zeta}{\omega}=\Re\sum_{\nu\in \Sigma}\zeta_\nu\overline{\omega}_\nu\mathpunct{.}
\end{equation}
 
Pour $\AA$ l'un des deux espaces $\C^S$ ou $\C^\Sigma$ et $\BB$ idem, 
\begin{enumerate}
\item si
$F:\AA\to \BB$ est $\R$-linéaire, son adjoint $F^*:\BB\to \AA$ est
défini par, $\forall a\in \AA, b\in \BB$,
$$
\pdtscalaire{a}{F^*.b}=\pdtscalaire{F.a}{b}\mathpunct{,}
$$
\item Si $F:\AA\to \BB$ est $\C$-linéaire,  son adjoint $F^*:\BB\to \AA$ est
défini par, $\forall a\in \AA, b\in \BB$,
$$
\pdthermitien{a}{F^*.b}=\pdthermitien{F.a}{b}\mathpunct{,}
$$
\item Une application $\C$-linéaire est aussi $\R$-linéaire, les deux
  notions d'adjoint précédentes coincident.
\end{enumerate}

Rappelons les formules de projection usuelles dans un espace hermitien.
\begin{lemma}\label{lemme:projection-hermitienne}
Soit $\HH$ un $\C$-espace vectoriel muni d'un produit hermitien
$\pdthermitien{.}{.}$, pour $u\in \HH$, on note $\|u\|$
 la norme euclidienne associée sur le
$\R$-espace vectoriel $\HH$ sous-jacent.
  \begin{enumerate}
  \item Soit $u\in \HH$ un vecteur non nul et considérons  le
  sous-espace vectoriel $u^\bot:=E_u=\{v\in \HH,\, \pdthermitien{v}{u} =0\}$ ainsi que
  $\pi_u$, la projection orthogonale aux sens hermitien et euclidien sur $E_u$. On a pour $z\in \HH$,
$$
\pi_u.z = z- \pdthermitien{z}{u}\frac{u}{\|u\|^2}
$$
\item Soit $u, v\in \HH$ deux vecteurs non nuls, $\pdthermitien{u}{v}=0$ et considérons
  $\pi_{u,v}$, la projection orthogonale aux sens hermitien et
  euclidien sur $E_u\cap E_v$. On a pour $z\in \HH$,
$$
\pi_{u,v}.z = z- \pdthermitien{z}{u}\frac{u}{\|u\|^2}-\pdthermitien{z}{v}\frac{v}{\|v\|^2}
$$
  \end{enumerate}
\end{lemma}

\subsection{Quelques applications linéaires}
Nous adoptons comme convention de désigner par la même lettre des
opérateurs ou des objets agissant sur ou appartenant à des espaces
différents pourvu que ceux-ci aient une définition similaire dans leur
cadre propre. L'examen des espaces en jeu (le typage) doit lèver les ambiguités sur 
l'opérateur ou l'objet en question.

Dans cet esprit, notons $\carac$ les vecteurs $(1)_{n\in S}\in\C^S$ et
$\carac=(1)_{\nu\in  E}\in\C^\Sigma$, 
\begin{equation}
  \label{eq:def-C0}
\C_0^S = \carac^\bot=\{z\in\C^S,\,\sum_n z_n=0\},\,
\C_0^\Sigma = \carac^\bot=\{\zeta\in\C^\Sigma,\,\sum_\nu \zeta_\nu=0\}
\mathpunct{.}
\end{equation}

Si $\pi_0$ est la projection hermitienne de $\C^S$ sur $\C_0^S$ ou de
$\C^\Sigma$ sur $\C_0^\Sigma$, on a, pour $z\in \C^S$, $\zeta\in\C^\Sigma$,
$$
\pi_0.z = z -(\frac1N \sum_n z_n)\carac,\,\pi_0.\zeta = \zeta -(\frac1N \sum_\nu \zeta_\nu)\carac
$$

Toujours dans le même esprit, 
Pour $a\in\C^S$, resp. $\alpha\in\C^\Sigma$, on définit
$\opdiag{a}$, resp. $\opdiag{\alpha}$ l'opérateur diagonal dont la
diagonale est $a$, resp. $\alpha$,
\ie, pour $z\in\C^S$, resp. $\zeta\in\C^\Sigma$,
$\opdiag{a}.z \in \C^S$, resp. $\opdiag{\alpha}.\zeta \in \C^\Sigma$
avec
$$
\forall n\in S,\,(\opdiag{a}.z)_n = a_n.z_n
\text{ et }
\forall\nu\in\Sigma,\,(\opdiag{\alpha}.\zeta)_\nu = \alpha_\nu.\zeta_\nu$$
On a 
$\opdiag{\alpha}^* = \opdiag{\overline{\alpha}}$

On définit  $C:\C^S\to \C^S$ et $C: \C^\Sigma \to \C^\Sigma$ les applications $\R$-linéaires de conjugaison
  définies par $C.(z_n)=(\overline{z_n})$ et
  $C.(\zeta_\nu)=(\overline{\zeta}_\nu)$. On a $C^*=  C$.

\subsection{L'opérateur de dérivation $D$, l'opérateur de milieu $M$}

On définit $D:\C^S\to \C^\Sigma$, $D:\C^\Sigma \to \C^S$, $M:\C^S\to \C^\Sigma$,
$M:\C^\Sigma \to \C^S$ par les formules, pour $z\in\C^S$, $\zeta\in\C^\Sigma$
\begin{eqnarray}
  \label{eq:def-D-M}
  (D.z)_{\nu}&:=& z_{\nu+\demi}-z_{\nu-\demi}\mathpunct{,}\,
  (D.\zeta)_{n}:= \zeta_{n+\demi}-\zeta_{n-\demi} \\
  (M.z)_{\nu}&:=& \frac12(z_{\nu+\demi}+z_{\nu-\demi})\mathpunct{,}\,
  (M.\zeta)_{n}:=\frac12(\zeta_{n+\demi}+\zeta_{n-\demi}) 
\end{eqnarray}
Le lemme suivant est immédiat, nous donnons les preuves de ces formules dans le paragraphe \ref{sec:preuve-lemme:prop-elem-D-M}.
\begin{lemma}\label{lemme:prop-elem-D-M}
  \begin{enumerate}
  \item $D^*=-D$, $M^*=M$,
\item $D.\pi_0=\pi_0.D=D$,
\item $M.\pi_0=\pi_0.M$,
\item $MD=DM$ et $(MD.z)_n = \frac12(z_{n+1}-z_{n-1})$, $(MD.\zeta)_\nu = \frac12(\zeta_{\nu+1}-\zeta_{\nu-1})$,
\item (Formule de \ancetre{Leibniz}), pour $\alpha\in \C^S$ ou $\alpha\in\C^\Sigma$, 
$$D.\opdiag{\alpha} = \opdiag{M.\alpha}.D+\opdiag{D.\alpha}.M$$ 
  \end{enumerate}
\end{lemma}

\subsection{L'opérateur d'intégration $I$, l'opérateur $K$,
  intégration par parties}

$D$ est clairement un isomorphisme de $\C^S_0$ sur $\C^\Sigma_0$, appelons
son inverse $I$ et étendons linéairement $I$ à $\C^\Sigma$ en posant
$I.\carac = 0$.

On remarque, entre autres, que $DM=MD$, $MI=IM$. On
pose aussi
$$K=MI:\C^\Sigma\to\C^\Sigma\mathpunct{.}$$

De manière analogue, on définit $I:\C^S \to \C^\Sigma$ et
$K:\C^S\to\C^S$. 

Nous donnerons des formules explicites pour $I$ et $K$ dans le paragraphe \ref{sec:montee-gradient}.

On vérifie les faits énoncés dans le lemme suivant au paragraphe \ref{sec:preuve-lemme:prop-I-K}, l'intégration par parties étant vue comme une version duale de la formule de \ancetre{Leibniz}.
\begin{lemma}\label{lemme:prop-I-K}
  \begin{enumerate}
  \item $I^*=-I$, $K^*=-K$,
  \item\label{lemme:prop-I-K:inversion-D} $D.I=I.D=\pi_0$,
  \item $D.\opdiag{\alpha}.I = \opdiag{M.\alpha}.\pi_0 + \opdiag{D.\alpha}K$,
  \item (Intégration par parties) 
    $$I.\opdiag{\alpha}.D = \pi_0\opdiag{M.\alpha} - K.\opdiag{D.\alpha}\mathpunct{.}$$
  \end{enumerate}
\end{lemma}

\subsection{Quelques résolutions de systèmes linéaires.}\label{sec:resolution-sys}
Nous allons résoudre des systèmes linéaires d'un type particulier, se
présentant comme un ensemble d'équations linéaires en des inconnues
complexes et leurs conjuguées. A titre d'exemple introductif nous
retrouvons la formule donnant l'affixe du cercle circonscrit à trois
points non alignés du plan. 

\begin{proposition}\label{prop:centre-CC}
  Soit $u,v,w\in\C$ trois points non alignés. Posons
    $$c=c(u,v,w)=\frac{1}{w-u}\left(\frac{w-v}{\overline{w-v}}-\frac{u-v}{\overline{u-v}}\right)\mathpunct{.}$$
Le centre du cercle  circonscrit à $u,v,w$ est
    $$o=v+\frac{c}{|c|^2}$$ 
\end{proposition}

\begin{proof}
  Le centre $o$ du cercle circonscrit est l'intersection des
  médiatrices des segments $[u,v]$ et $[w,v]$. $t=o-v$ est donc
  solution du système
$$
\left\{
  \begin{array}{lcc}
(t-\frac{u-v}{2}).\overline{(u-v)} + \overline{(t-\frac{u-v}{2})}.(u-v)&=& 0 \\    
(t-\frac{w-v}{2}).\overline{(w-v)} + \overline{(t-\frac{w-v}{2})}.(w-v)&=& 0 \\    
  \end{array}
\right.
$$
La première équation étant l'expression de l'orthogonalité des
vecteurs $o-\frac{u+v}{2}$ et $u-v$, la seconde, celle de l'orthogonalité des
vecteurs $o-\frac{w+v}{2}$ et $w-v$.

Après simplification, en posant, $s=\overline{t}$, il s'agit de
résoudre le système linéaire d'inconnues $(t,s)$
$$
\left\{
  \begin{array}{lcc}
t.\overline{(u-v)} + s.(u-v)&=& |u-v|^2 \\    
t.\overline{(w-v)} + s.(w-v)&=& |w-v|^2 \\    
  \end{array}
\right.
$$
Ce sytème $2\times 2$ est de déterminant
$\Delta=\overline{(u-v)}.(w-v)-\overline{(w-v)}(u-v)$, non nul du fait
du non alignement de $u,v,w$ et les formules de \ancetre{Cramer}
donnent
$$
t=\frac{\overline{(u-v)}.(u-v).(w-v)-\overline{(w-v)}.(w-v)(u-v)}{\Delta}
=\frac{(u-v).(w-v).\overline{(u-w)}}{\overline{(u-v)}.(w-v)-\overline{(w-v)}(u-v)}
$$
et
$$
s=\overline{t}=\frac1{c}
$$
\end{proof}

Nous élaborons maintenant sur ce thème en augmentant le nombre de variables.

Le lemme suivant est un simple exercice d'algèbre linéaire élémentaire, on peut trouver sa démonstration au paragraphe \ref{sec:lemme:resolution-systeme}.
\begin{lemma}[Résolution de systèmes linéaires particuliers]\label{lemme:resolution-systeme}
  Soient $n$ et $p$ deux entiers naturels et soit $A$ une matrice à
  coefficients complexes\footnote{Pour une matrice, ou un vecteur $A$ à coefficients complexes, on note
$\overline{A}$ la matrice ou le vecteur dont les entrées sont les
conjuguées des entrées de $A$.}, $p$ lignes, $n$ colonnes.
On considère les ensembles suivants 
\begin{eqnarray*}
  \TT&=&\{T\in\C^n,\,A.T+\overline{A}.\overline{T}=0\},\\
  \TT_0&=&\{(T,S)\in\C^n\times\C^n,\,A.T+\overline{A}.S=0\},\\
  \TT_R&=&\{(T,S)\in\C^n\times\C^n,\,A.T+\overline{A}.S=0,T=\overline{S}\},\\
  \TT_I&=&\{(T,S)\in\C^n\times\C^n,\,A.T+\overline{A}.S=0,T=-\overline{S}\},
\end{eqnarray*}
alors 
\begin{enumerate}
\item $\TT_0$ est un $\C$-sev de $\C^n\times\C^n$, $\TT_R$ et $\TT_I$
  sont des $\R$-sev de $\C^n\times\C^n$, $\TT$ est un $\R$-sev de
  $\C^n$,
\item $\TT_0=\TT_R \oplus \TT_I$, au sens des $\R$-sous-espaces de $\C^n\times\C^n$, avec projections associées\footnote{On note $\pi_{R//I}$ la projection sur $\TT_R$ parallèlement $\TT_I$ et $\pi_{R//I}$ la projection sur $\TT_I$ parallèlement $\TT_R$}
$$\pi_{R//I}(T,S)=\frac12(T+\overline{S},\overline{T}+S),\,
\pi_{I//R}(T,S)=\frac12(T-\overline{S},-\overline{T}+S).$$

\item L'application $\TT_R\to\TT_I$, $(T,S)\mapsto (i.T,i.S)$ est un isomorphisme.
L'application $\TT_R\to \TT$, $(T,S)\mapsto T$ est
  un isomorphisme.
\item En particulier, $\dim_\R \TT = \dim_\R \TT_r = \frac12\dim_\R \TT_0=\dim_\C \TT_0$.
\item Si $J$ et $Q$ sont deux matrices telles que
$$\TT_0=\{(T,S)\in\C^n\times\C^n,T=J.S,Q.S=0\}$$
alors $$\TT=\{J.S+\overline{S},\, S\in\C^n, Q.S=0\}=(J+C)(\kernel{Q})$$
et $\dim_\R \TT= \dim_\C\kernel{Q}$.
\end{enumerate}
\end{lemma}

\subsection{Définition de $J_\alpha$.}\label{section:def-Jalpha}

  On pose, pour $\alpha\in(\S^1)^\Sigma$,
  \begin{equation}
    \label{eq:def-Jalpha}
J_\alpha = I.\opdiag{\alpha}.D    
  \end{equation}

L'application directe des lemmes \ref{lemme:resolution-systeme} et \ref{lemme:prop-I-K}.(\ref{lemme:prop-I-K:inversion-D}) donne
immédiatement\footnote{Il suffit de poser
  $\TT_0=\{(T,S)\in\C^S\times\C^S,\,\opdiag{\beta}D.T+\opdiag{\overline{\beta}}.D.S=0,
  \pdthermitien{T}{\carac}+\pdthermitien{S}{\carac}=0,\pdthermitien{T}{\carac}-\pdthermitien{S}{\carac}=0\}$ avec  $\beta_\nu$, une racine carrée de $\alpha_\nu$.}
\begin{lemma}\label{lemme:solutions-systeme-tangent} Soit $\alpha=(\alpha_\nu)\in(\S^1)^\Sigma$, 
  \begin{enumerate}
  \item L'ensemble des solutions du  système $\R$-linéaire
$$
\left\{ 
  \begin{array}{ll}
t_{\nu{+\demi}}-t_{\nu{-\demi}}+\alpha_\nu(\overline{t}_{\nu{+\demi}}-\overline{t}_{\nu{-\demi}})=0
&    \forall \nu\in \Sigma \\
\sum_n t_n = 0
  \end{array}
\right.
$$
d'inconnue $t=(t_n)\in\C^S$, est le $\R$-sev de $\C^S$
\begin{eqnarray*}
\TT_\alpha&=&\{ 
t = -J_\alpha.s+\overline{s}, s=(s_n)\in\C^S,\, \sum_n
s_n=0,\sum_n (\alpha_{n{+\demi}}-\alpha_{n{-\demi}})s_n =0
\}\\
&=& (-J_\alpha+C)(\carac^\bot \cap (D.\overline{\alpha})^\bot)
\end{eqnarray*}  
\item La dimension de $\TT_\alpha$ (sur $\R$) est la dimension (sur $\C$) de 
$$\SS_\alpha:=\{s\in\C^S,\,\sum_n
s_n=0,\sum_n (\alpha_{n{+\demi}}-\alpha_{n{-\demi}})s_n =0\}=\carac^\bot \cap (D.\overline{\alpha})^\bot$$
\begin{enumerate}
\item Elle vaut $N-1$ si les $\alpha_\nu$ sont tous égaux,
\item elle vaut $N-2$ sinon.
\end{enumerate}
  \end{enumerate}
\end{lemma}

\section{Une représentation paramétrique de l'espace tangent $\TT_z$}\label{sec:rep-param-Tz}
Pour un polygone $z\in \PP_0(\ell)$, posons $\zeta=D.z$ et $\alpha\in
\C^\Sigma$ défini par 
\begin{equation}
  \label{eq:def-alpha}
\alpha_\nu = \frac{\zeta_\nu}{\overline{\zeta}_\nu},\,\forall \nu\in \Sigma  
\end{equation}
$\alpha$ est bien défini car, pour $\nu\in \Sigma$, $|\zeta_\nu|=\ell_\nu >
0$.

Les fonctions $L^2_\nu$ ont été définies en \eqref{eq:def-Lnu}. Pour
$\nu\in\Sigma$, en utilisant les règles de calcul différentiel rappelées au paragraphe \ref{sec:calc-diff-rappels}, comme pour $\nu\in\sigma$, 
$$
L^2_\nu(z)=|z_{\nu+\demi}-z_{\nu-\demi}|^2= (z_{\nu+\demi}-z_{\nu-\demi}).\overline{(z_{\nu+\demi}-z_{\nu-\demi})}
\mathpunct{,}
$$
la différentielle de $L^2_\nu$ en $z$ vaut
\begin{equation}
  \label{eq:diff-L}
d_z L^2_\nu = \overline{z_{\nu{+\demi}}-z_{\nu{-\demi}}}(dz_{\nu{+\demi}}-dz_{\nu{-\demi}})+(z_{\nu{+\demi}}-z_{\nu{-\demi}})(d\overline{z}_{\nu{+\demi}}-d\overline{z}_{\nu{-\demi}})  
\end{equation}

Comme 
$$\PP_0(\ell)=\{z\in\C^S, L^2_\nu(z)=\ell_\nu^2,\forall\nu\in \Sigma\}\cap \carac^\bot$$
Son espace tangent en $z$, au sens réel algébrique,  est 
\begin{equation}
  \label{eq:cartesienne-Tz}
\TT_z := \{(t_n)\in\C^S,\, \overline{z_{\nu{+\demi}}-z_{\nu{-\demi}}}(t_{\nu{+\demi}}-t_{\nu{-\demi}})+(z_{\nu{+\demi}}-z_{\nu{-\demi}})(\overline{t}_{\nu{+\demi}}-\overline{t}_{\nu{-\demi}})=0,\forall\nu\in \Sigma\}\cap\carac^\bot
\end{equation}
D'après le lemme \ref{lemme:solutions-systeme-tangent}, en posant 
$$
\SS_z := \{ s=(s_n)\in\C^S,\, \sum_n
s_n=0,\sum_n (\alpha_{n{+\demi}}-\alpha_{n{-\demi}})s_n =0
\}
\mathpunct{,}
$$ 
on a
\begin{equation}
  \label{eq:parametrage-Tz}
\TT_z = \{ -J_\alpha.s+\overline{s}, s=\SS_z\}=(-J_\alpha+C)(\SS_z)
\end{equation}
où $J_\alpha$ est défini dans le paragraphe \ref{section:def-Jalpha}.
$\TT_z$ est de dimension réelle $N-1$ si les $\alpha_\nu$ sont tous égaux,
$N-2$ sinon.

Le fait que les $\alpha_\nu$ soient tous égaux équivaut au fait qu'il
existe un nombre complexe de module $1$, $e^{i\theta_0}$ tel que, pour
tout $\nu\in \Sigma$, $z_{\nu{+\demi}}-z_{\nu{-\demi}} = \epsilon_\nu\ell_\nu
e^{i\theta_0}$ avec $\epsilon_\nu=\pm1$. Ceci équivaut au fait que
d'une part 
\begin{equation}
  \label{eq:condition-P-var-diff}
  \sum_\nu \epsilon_\nu\ell_\nu = 0
\end{equation}
et d'autre part les point $z_n$ sont alignés sur une droite.

On résume ce qui vient d'être dit sous la forme suivante.
\begin{proposition}
  \label{prop:Pl-variete}
Dans tous les cas, si $\ell$ est admissible, les polygones {\it  colinéaires}\footnote{Un polygone est colinéaire si tous les points sont alignés.}
sont les points singuliers de $\PP_0(\ell)$.
  \begin{enumerate}
  \item Si, pour tout $\epsilon\in\{-1,+1\}^\Sigma$, $\sum_\nu
    \epsilon_\nu.\ell_\nu\not=0$, alors tout point
    $z$ de $\PP_0(\ell)$ est \emph{régulier}. $\PP_0(\ell)$ est une sous-variété
    différentiable réelle de $\C^S$ de dimension $N-2$. 
\item Si, par contre, il existe un $\Sigma$-uplet
  $(\epsilon_\nu)=(\pm1)_\nu$ tel que $\sum_\nu
    \epsilon_\nu.\ell_\nu=0,$ 
il y a au moins un point singulier dans $\PP_0(\ell)$.
\end{enumerate}
\end{proposition}

Examinons ce qui arrive lorsqu'on  passe au quotient sous l'action
diagonale de $\S^1$, $(e^{i\theta},z)\mapsto e^{i\theta}.z$. 

Fixons $z\in\PP_0(\ell)$. La courbe $\theta\mapsto e^{i\theta}.z$,
contenue dans $\PP_0(\ell)$, a
pour vecteur tangent en $\theta=0$ le vecteur $i.z$. Celui-ci
appartient donc à l'espace tangent $\TT_z$. Dans le paramétrage
\eqref{eq:parametrage-Tz}, ce vecteur est l'image du vecteur
$s=\frac{i}{2}\overline{z}\in \SS_z$.

L'espace tangent à $\MM(\ell)$ en $z$ est  naturellement isomorphe à $\TT_z / \vect[\R]{i.z}$.

\section{Preuve du théorème }\label{sec:preuve-thm}

Nous avons démontré que les points singuliers de
$\PP_0(\ell)$ sont les polygones colinéaires. Pour conclure le
théorème, on se place en un polygone $z\in\PP_0(\ell)$,  non
colinéaire. On démontre alors qu'il est critique pour l'aire sur
$\PP_0(\ell)$ si et seulement si il est cocyclique.

La démonstration est basée sur le calcul de la norme d'opérateur de la
restriction de $d_z A$ à $\TT_z$ après avoir équipé cet espace\footnote{Cette
famille de normes euclidiennes définit une structure Riemannienne de $\PP_0(\ell)$
et nous calculerons en paragraphe \ref{sec:montee-gradient} le gradient de $A$ sur
$\PP_0(\ell)$ relativement à cette structure.} d'une norme euclidienne adéquate. 

Le polygone $z$ est critique pour $A$ sur $\PP_0(\ell)$ si et
seulement si cette norme est nulle.

Définissons sur $\SS_z \subset\C^S$ la fonction 
\begin{equation}
  \label{eq:def-AA}
\AA_z(s) = d_z A . (-J_\alpha.s+\overline{s}) =d_z A . (-J_\alpha+C).s 
\end{equation}

$\AA_z$ est une forme linéaire réelle sur $\SS_z$. Il existe donc un vecteur
$a_z \in\SS_z$ tel que pour tout $s\in\SS_z$,
\begin{equation}
  \label{eq:def-az}
\AA_z(s) = \pdtscalaire{a_z}{s}  
\end{equation}

La norme euclidienne usuelle de ce vecteur est
\begin{eqnarray*}
\|a_z\| &=& \sup_{s\in\SS_z,\|s\|\leq 1} \pdtscalaire{a_z}{s} =\sup_{s\in\SS_z,\|s\|\leq 1} \AA_z(s) \\
 &=& \sup_{w\in\TT_z, \|w\|_z\leq 1} d_z A . w
\end{eqnarray*}
où $\|.\|_z$ est la norme euclidienne sur
$\TT_z$ dont la boule unité est l'ellipsoïde $(-J_\alpha+C)(\B_1\cap
\SS_z)$ où $\B_1$ est la boule unité euclidienne usuelle dans $\C^S$.

Plus précisémment, on a, pour $w\in \TT_z=(-J_\alpha+C)(\SS_z)$ 
\begin{equation}
  \label{eq:def-struct-riem}
  \|w\|_z = \inf\{\|s\|,\, s\in\SS_z,\, w=(-J_\alpha+C).s\}
\end{equation}

Il est clair que $z$ est critique pour $A$ sur $\PP_0(\ell)$ si et seulement si $\|a_z\|=0$.

Explicitons maintenant le vecteur $a_z$ et calculons $\|a_z\|$.

On a 
\begin{eqnarray*}
  \AA_z(s)&=& 
\frac12\sum_k
\Im((z_{k+1}-z_{k-1}).\overline{(-J_\alpha.s+\overline{s})_k})
= -\pdtscalaire{i.D.M.z}{(-J_\alpha+C).s}\\
&=& -\pdtscalaire{(-J_\alpha^*+C)i.D.M.z}{s}
= \pdtscalaire{i(J_\alpha^*+C)D.M.z}{s}
\end{eqnarray*}

Calculons le vecteur $\tilde{a}_z = i(J_\alpha^*+C)D.M.z \in\C^S$.

\begin{itemize}
\item D'une part, $$(i.C.M.D.z)_n =\frac{i}{2}\overline{z_{n+1}-z_{n-1}}$$
\item d'autre part, en utilisant les diverses règles régissant $I,M,D,\opdiag{\alpha},\pi_0$, 
\begin{eqnarray*}
  J_\alpha^*.D.M &=& D.\opdiag{\overline{\alpha}}.I.D.M =
D.\opdiag{\overline{\alpha}}.M.\pi_0
\end{eqnarray*}
On a $\pi_0. z = z$ et, pour $n\in S$,
\begin{eqnarray*}
  (D.\opdiag{\overline{\alpha}}.M.z)_n&=&
  (\opdiag{\overline{\alpha}}.M.z)_{n{+\demi}}-(\opdiag{\overline{\alpha}}.M.z)_{n{-\demi}}= 
\frac12(\overline{\alpha}_{n{+\demi}}(z_{n+1}+z_{n})-\overline{\alpha}_{n{-\demi}}(z_{n}+z_{n-1}))\\
&=&\frac12(\overline{\alpha}_{n{+\demi}}(z_{n+1}-z_{n})-\overline{\alpha}_{n{-\demi}}(z_{n-1}-z_{n}))+z_n(\overline{\alpha}_{n{+\demi}}-\overline{\alpha}_{n{-\demi}})\\
&=&\frac12\overline{z_{n+1}-z_{n-1}}+z_n(\overline{\alpha}_{n{+\demi}}-\overline{\alpha}_{n{-\demi}})
\end{eqnarray*}
Remarquons que pour obtenir cette dernière ligne, nous avons utilisé que
$$\alpha_{n{+\demi}}=\frac{z_{n+1}-z_n}{\overline{z_{n+1}-z_n}} \text{ et }
\alpha_{n{-\demi}}=\frac{z_{n-1}-z_n}{\overline{z_{n-1}-z_n}}$$
\item Finalement,
  \begin{equation}
    \label{eq:formule-atildez}
\forall n\in S,\,(\tilde{a}_z)_n = 
 i\left(  \overline{z_{n+1}-z_{n-1}}+z_n\overline{\alpha_{n{+\demi}}-\alpha_{n{-\demi}}} \right)
  \end{equation}

\end{itemize}

Le vecteur $u=(\overline{\alpha_{n{+\demi}}-\alpha_{n{-\demi}}})_{n\in
  S}$ est non nul car le polygone $(z_n)$ n'est pas colinéaire.

Le vecteur $a_z$ est la projection orthogonale de $\tilde{a}_z$ sur
$\SS_z=\carac^\bot \cap u^\bot $. En appliquant le lemme
\ref{lemme:projection-hermitienne}, en remarquant que $\pdthermitien{\tilde{a}_z}{\carac}=0$,
\begin{equation}
  \label{eq:formule-az}
  a_z
  =\tilde{a}_z-\frac{\pdthermitien{\tilde{a}_z}{(\overline{\alpha_{n{+\demi}}-\alpha_{n{-\demi}}})_n}}{\sum_n|\alpha_{n{+\demi}}-\alpha_{n{-\demi}}|^2}(\overline{\alpha_{n{+\demi}}-\alpha_{n{-\demi}}})_n 
\end{equation}
Par \ancetre{Pythagore},
\begin{eqnarray*}
  \|a_z\|^2&=& \frac{\|\tilde{a}_z\|^2\|(\overline{\alpha_{n{+\demi}}-\alpha_{n{-\demi}}})_n \|^2-|\pdthermitien{\tilde{a}_z}{(\overline{\alpha_{n{+\demi}}-\alpha_{n{-\demi}}})_n}|^2}{\sum_n|\alpha_{n{+\demi}}-\alpha_{n{-\demi}}|^2}
\end{eqnarray*}
Par le cas d'égalité de l'inégalité de \ancetre{Cauchy}--\ancetre{Schwarz}, $\|a_z\|^2=0$ si et seulement
si $\tilde{a}_z$ et $(\overline{\alpha_{n{+\demi}}-\alpha_{n{-\demi}}})_n$ sont
$\C$-colinéaires, \ie, il existe $o\in\C$ tel que, pour tout $n\in
S$,
$$
\overline{z_{n+1}-z_{n-1}}+z_n\overline{\alpha_{n{+\demi}}-\alpha_{n{-\demi}}}
= o.\overline{\alpha_{n{+\demi}}-\alpha_{n{-\demi}}}
$$
et donc, pour tout $n\in S$,
\begin{itemize}
\item soit $\alpha_{n{+\demi}}-\alpha_{n{-\demi}}=0$ auquel cas $z_{n+1}=z_{n-1}$,
\item soit
  $z_n+\overline{\left(\frac{z_{n+1}-z_{n-1}}{\alpha_{n{+\demi}}-\alpha_{n{-\demi}}}\right)}=o$,
  $o$ est donc, par la proposition \ref{prop:centre-CC} le centre du cercle circonscrit aux points
  $z_{n-1},z_n, z_{n+1}$.  
\end{itemize}

Nous prouvons maintenant que dans une telle situation, tous les points
sont sur un même cercle. Posons
$$Z:=\{n\in S,\,\alpha_{n+\demi}=\alpha_{n-\demi}\} = \{n\in
S,\,z_{n+1}=z_{n-1}\}$$

Par hypothèse de non colinéarité de $z$, il existe un point $n_0\not\in Z$. En définissant $\Gamma$, le
cercle de centre $o$ de rayon $R=|z_{n_0}-o|$, on montre alors facilement\footnote{
En posant comme hypothèse de récurrence au rang $p$: $z_{n_0+p}$, $z_{n_0+p+1}\in
\Gamma$. Comme $n_0\not\in Z$, cette hypothèse est vraie au rang
$p=0$. Supposons la vraie au rang $p$ et montrons qu'elle est vraie au
rang $p+1$. Il s'agit de montrer que $z_{n_0+p+2}\in \Gamma$. A
ceci, deux possibilités, soit $p+1\in Z$ et à ce moment
$z_{n_0+p+2}=z_{n_0+p}\in\Gamma$, soit $p+1\not\in Z$ et à ce moment
$z_{n_0+p}$, $z_{n_0+p+1}$ et $z_{n_0+p+2}$ sont sur un même cercle
de centre $o$ qui ne peut être que $\Gamma$.  } par récurrence, que tous les $z_n$, $n\in S$ appartiennent à ce
cercle. 

Le théorème \ref{thm:pt-critiques-aire} est donc démontré.

\section{Multiplicateurs de \ancetre{Lagrange}}\label{sec:lagrange}

\subsection{Leur valeur}\label{ssec:valeur}
Posons, pour $z=(z_n)\in \C^S$, $\lambda=(\lambda_\nu)\in\R^\Sigma$,
$\mu\in\C$,
$$
\LL(z,\lambda,\mu) =2.A(z)+\frac12\sum_\nu
\lambda_\nu(L^2_\nu(z)-\ell_\nu^2)+\Re(\mu\sum_n \overline{z}_n),
$$
le lagrangien associé au problème de la minimisation de $A$ sous la
contrainte $z\in \PP_0(\ell)$.

Il est classique, c'est le théorème des multiplicateurs de
\ancetre{Lagrange}, que, 
à l'exception des point singuliers de $\PP_0(\ell)$, 
les points critiques de $\LL$ sur
$\C^S\times\R^\Sigma\times\C$ sont exactement
les points critiques de $A$ sur $\PP_0(\ell)$.

Calculons donc les multiplicateurs de \ancetre{Lagrange} associés à chaque point
critique de $A$. On a
\begin{eqnarray*}
  d\LL&=& \Re\left( \sum_n\left\{-i (z_{n+1}-z_{n-1})+\lambda_{n-\demi}(z_n-z_{n-1})-\lambda_{n+\demi}(z_{n+1}-z_{n})+\mu\right\}d\overline{z}_n\right)\\
&&
+\frac12\sum_\nu(L^2_\nu(z)-\ell_\nu^2)d\lambda_\nu
+\Re(\sum_n\overline{z}_n d\mu) 
\end{eqnarray*}

$(z,\lambda,\mu)$ est critique pour $\LL$ si et seulement si
\begin{enumerate}
\item $\sum_n\overline{z}_n=0$,
\item $L^2_\nu(z)-\ell_\nu^2=0,\,\forall\nu\in\Sigma$,
\item
  $-i(z_{n+1}-z_{n-1})+\lambda_{n-\demi}(z_n-z_{n-1})-\lambda_{n+\demi}(z_{n+1}-z_{n})+\mu=0,\,\forall  n\in S$.
\end{enumerate}

Les deux premières conditions marquent l'appartenance de $z$ à
$\PP_0(\ell)$, la dernière condition permet de calculer les
multiplicateurs $\lambda,\mu$.

Supposons que $z$ est point critique de $A$ sur
$\PP_0(\ell)$. Il n'est pas colinéaire et est cocyclique. Il existe alors $o\in\C$, $R>0$, $\theta=(\theta_n)\in\R^S$ tels que 
$z_n=o+Re^{i\theta_n}$. On donne la valeur de $\lambda$ en
fonction de la famille d'angles $\theta$.

La dernière condition s'écrit en termes matriciels, en insérant le point $o$,
$$
D.\opdiag{\lambda}.D.z = -2iD.M.(z-o\carac)+\mu.\carac
$$
Ceci montre, d'une part que $\mu=0$ et qu'il existe $c\in\C$ tel que
$$
\opdiag{\lambda}.D.z = -2iM.(z-o\carac)+c\carac
$$
et donc, il existe $c\in\C$ tel que pour tout $\nu\in\Sigma$,
$$
\lambda_\nu(z_{\nu+\demi}-z_{\nu-\demi})=-i(z_{\nu+\demi}-o + z_{\nu-\demi}-o)+c
$$

On doit donc avoir
$$
2i\lambda_\nu \sin \frac{\theta_{\nu+\demi}-\theta_{\nu-\demi}}{2}=-i\cos
\frac{\theta_{\nu+\demi}-\theta_{\nu-\demi}}{2} + c.e^{-i\frac{\theta_{\nu+\demi}+\theta_{\nu-\demi}}{2}}
$$
Comme $\lambda_\nu\in\R$, si $c\not=0$, on doit donc avoir, pour tous
$\nu$,$\nu'\in\Sigma$,
$$\frac{\theta_{\nu+\demi}+\theta_{\nu-\demi}}{2}=\frac{\theta_{\nu'+\demi}+\theta_{\nu'-\demi}}{2}[\pi]
$$
Pour continuer l'analyse, on peut supposer que pour un certain
$\nu_0$, le polygone a été normalisé de sorte que
$\theta_{\nu_0+\demi}+\theta_{\nu_0-\demi}=0[2\pi]$. Cela force alors, pour tout $\nu\in\Sigma$,
$$
\theta_{\nu+\demi}=-\theta_{\nu-\demi}[2\pi]
$$
\begin{itemize}
\item Si $N$ est impair, on a alors $\theta_{\nu+\demi}=0[2\pi]$, pour tout
  $\nu$, ce qui est impossible car alors $z$ serait réduit à un point.
\item Si $N$ est pair, cela force à se retrouver dans une situation
  colinéaire, ce que l'on a exclu pour $z$.
\end{itemize}

En supposant donc que $z$ est point critique de $A$ sur $\PP_0(\ell)$, il vient $c=0$. Remarquons par
ailleurs que 
$$L^2_\nu(z)=|z_{\nu+\demi}-z_{\nu-\demi}|^2 = \ell_\nu^2 = 4R^2\sin^2
\frac{\theta_{\nu+\demi}-\theta_{\nu-\demi}}{2}\not=0$$
et donc
$$
\forall \nu\in S,\,\lambda_\nu = - \cot
\frac{\theta_{\nu+\demi}-\theta_{\nu-\demi}}{2}=-i.\frac{z_{\nu+\demi}-o + z_{\nu-\demi}-o}{z_{\nu+\demi}-z_{\nu-\demi}}\mathpunct{.}
$$

\subsection{Points critiques avec un côté libre}

Fixons une arête $\nu_0$. On peut considérer le problème de trouver
les points critiques de $A$ sous la contrainte que toutes les
longueurs $L_\nu$ soient imposées, à l'exception de $L_{\nu_0}$.

Cela nous conduit à rechercher les points critiques du lagrangien
$$
\LL_{\nu_0}(z,\lambda,\mu) =2.A(z)+\frac12\sum_{\nu\not=\nu_0}
\lambda_\nu(L^2_\nu(z)-\ell_\nu^2)+\Re(\mu\sum_n \overline{z}_n),
$$

Pour abréger, on dira que $z$ est point critique du lagrangien $\LL$
ou $\LL_{\nu_0}$, s'il existe des paramètres complémentaires
$\lambda,\mu$... tels que $(z,\lambda,\mu)$ est point critique de
$\LL$, $\LL_{\nu_0}$...

Les points critiques $z$ de $\LL_{\nu_0}$ sont ceux de $\LL$ avec la
contrainte supplémentaire que $\lambda_{\nu_0}=0$. De tels $z$ sont
cocycliques s'ils sont non alignés. Notons $o$ le centre du cercle
circonscrit à $z$. 
En reprenant la
formule trouvée pour $\lambda_{\nu_0}$, il est clair que, pourvu que $z_{\nu_0+\demi}\not=z_{\nu_0-\demi}$,  
$$
\lambda_{\nu_0}=0 \ssi \frac12 (z_{\nu_0+\demi}+z_{\nu_0-\demi})= o
\mathpunct{,}
$$
ce qui équivaut au fait que le segment $[z_{\nu_0-\demi},z_{\nu_0+\demi}]$ est un diamètre du cercle
    circonscrit à $z$.

Ce résultat, dû à \ancetre{G. Khimshiashvili} et \ancetre{D. Siersma}, \cite{KS2013},  est à rapprocher de \cite{LEGENDRE14}, Proposition IV, p.132.

\subsection{Points critiques à périmètre imposé et variantes}

On peut considérer le problème de trouver
les points critiques de $A$ sous la contrainte que la somme des carrés des
longueurs, \ie $\sum_\nu L^2_\nu$ soit imposée. 
Cela nous conduit à rechercher les points critiques du lagrangien

$$
\LL_{\PP_2}(z,\tilde{\lambda},\mu) =2.A(z)+\frac12\tilde{\lambda}.\sum_{\nu}
(L^2_\nu(z)-\ell_\nu^2)+\Re(\mu\sum_n \overline{z}_n),
$$

Les points critiques de ce lagrangien se déduisent de ceux de $\LL$ avec la
contrainte supplémentaire que $\forall
\nu\in\Sigma,\lambda_{\nu}=\tilde{\lambda}$. Cette condition se réécrit
$$
\forall \nu,\nu'\in\Sigma,\, \theta_{\nu+\demi}-\theta_{\nu-\demi}=\theta_{\nu'+\demi}-\theta_{\nu'-\demi}[2\pi]
$$
où $(z_n)=(o+R.e^{i\theta_n})$. 

Pour $n\in S$, en prenant $\nu=n+\demi$, $\nu'=(n-1)+\demi$, on a donc
$$
\forall n\in S,\, \theta_{n+1}-\theta_{n}=\theta_{n}-\theta_{n-1}[2\pi]
$$
et ceci équivaut à l'existence d'angles $\theta_0$ et
$\alpha=0[\frac{2\pi}{N}]$ tels que, en ayant fixé un sommet $n_0\in
S$, on ait
$$
\forall k \in\ZnZ,\, z_{n_0+k} =o +R.e^{i.\theta_0}.e^{i.k.\alpha}
$$
Cela implique que les $N$ sommets $z_n$ sont les sommets d'un
polygone régulier\footnote{Ce qui n'équivaut pas au fait que $z$ soit
  le $N$-gone régulier au sens usuel du terme. Les polygones croisés
  sont autorisés, par exemple, prendre pour $N=8$,
  $\alpha=\frac{3\pi}{4}$. D'autres cas sont possibles, par exemple le
carré parcouru deux fois dans le cas $N=8$, $\alpha=\frac{\pi}{2}$.}

Le problème, en un sens plus classique--on est alors dans les conditions de la Proposition VII de \cite{LEGENDRE14}--, de trouver
les points critiques de $A$ sous la contrainte que le périmètre, \ie
$P(z)=\sum_\nu \sqrt{L^2_\nu(z)}$ soit imposé, se règle de manière
similaire.

En posant $\psi$, la fonction racine carrée, la méthode nous conduit à rechercher les points critiques du lagrangien
$$
\LL_{\PP_1}(z,\tilde{\lambda},\mu) =2.A(z)+\frac12\tilde{\lambda}.\sum_{\nu}
(\psi(L^2_\nu(z))-\psi(\ell_\nu^2))+\Re(\mu\sum_n \overline{z}_n),
$$

Ceux-ci sont les points critiques du lagrangien $\LL^\psi$ défini par
$$
\LL^\psi(z,\lambda,\mu) =2.A(z)+\frac12\sum_\nu
\lambda_\nu(\psi(L^2_\nu(z))-\psi(\ell_\nu^2))+\Re(\mu\sum_n \overline{z}_n),
$$
avec la contrainte supplémentaire que $\forall
\nu\in\Sigma,\lambda_{\nu}=\tilde{\lambda}$. 

Un polygone $z$, non colinéaire, de côtés de longueur $>0$, est
critique pour ce lagrangien si et seulement si il est cocyclique et un
calcul similaire à celui en \ref{ssec:valeur} (on conserve les notations
de ce calcul) montre qu'alors les
multiplicateurs de \ancetre{Lagrange} $(\lambda_\nu)$ vérifient 
$$
\forall \nu\in\Sigma,\, \lambda_\nu.\psi'(L^2_\nu(z))= - \cot
\frac{\theta_{\nu+\demi}-\theta_{\nu-\demi}}{2}
$$
Comme $\psi'(x)=\frac1{2\sqrt{x}}$ et
$\sqrt{L^2_\nu(z)}=2R\left|\sin\frac{\theta_{\nu+\demi}-\theta_{\nu-\demi}}{2}\right|$, on obtient
que
$$
\forall \nu\in\Sigma,\, \lambda_\nu=
-4R\signe(\sin\frac{\theta_{\nu+\demi}-\theta_{\nu-\demi}}{2}) \cos \frac{\theta_{\nu+\demi}-\theta_{\nu-\demi}}{2}
$$

Pour deux angles $\theta,\theta'$, on a
$$
\signe(\sin\theta)\cos\theta = \signe(\sin\theta')\cos\theta' \ssi \theta=\theta'[\pi] 
$$

Un  polygone $z=(z_n) =(o+R.e^{i\theta_n})$ est donc critique pour
le lagrangien à périmètre constant si et seulement si
$$
\forall \nu,\nu'\in\Sigma,\, 
\theta_{\nu+\demi}-\theta_{\nu-\demi} = \theta_{\nu'+\demi}-\theta_{\nu'-\demi} [2\pi]
$$
On obtient donc les mêmes points critiques que pour $\LL_{\PP_2}$.

\section{Montée de gradient}\label{sec:montee-gradient}

\subsection{Gradient de l'aire} 
Pour simplifier la discussion, on se place en $\ell$ tel que
$\PP_0(\ell)$ est une variété différentielle et donc ne contient pas
de polygones colinéaires. La discussion peut aussi se faire dans le
cas général, en excluant les polygones colinéaires, singuliers.

La démonstration du théorème principal fait apparaitre un gradient
pour la fonctionnelle d'aire sur l'espace $\PP_0(\ell)$. 

En effet, en reprenant l'équation \eqref{eq:def-struct-riem}, on
définit en chaque $z\in\PP_0(\ell)$ une norme euclidienne $\|.\|_z$
sur $\TT_z$, et donc un produit scalaire $\pdtscalaire{.}{.}_z$ sur $\TT_z$,
ce qui définit une structure
Riemanienne sur $\PP_0(\ell)$. 

Détaillons un peu ce point : Notons, pour alléger,
$J=-J_\alpha+C$. $J:\C^S \to \C^S$. On a $\TT_z = J(\SS_z)$.

Notons $\tilde{J}:\kernel{J}^{\bot}\cap \SS_z \to \TT_z$, la restriction
de $J$. $\tilde{J}$
est  un isomorphisme de $\R$-espaces vectoriels.

Si $w\in\TT_z$, $s\in\SS_z$ tel que $J.s=w$ et $\tilde{s}\in
\kernel{J}^{\bot}\cap \SS_z$, $\tilde{s}=\tilde{J}^{-1}.w$ alors 
$s-\tilde{s} \in \kernel{J}\cap\SS_z$ et, par orthogonalité, 
$$
\|s\|=\sqrt{\|\tilde{s}\|^2+\|s-\tilde{s}\|^2} \geq \|\tilde{s}\|
$$
En résumé,
$$
\|w\|_z=\inf\{\|s\|,\, s\in\SS_z,\, w=(-J_\alpha+C).s\}=\|\tilde{s}\|=\|\tilde{J}^{-1}w\|
$$

Le produit scalaire $\pdtscalaire{.}{.}_z$ est donc défini sur $\TT_z$
par la formule
$$
\forall w,w'\in\TT_z,\, \pdtscalaire{w}{w'}_z=\pdtscalaire{\tilde{J}^{-1}w}{\tilde{J}^{-1}w'}=\pdtscalaire{w}{(\tilde{J}.\tilde{J}^*)^{-1}w'}
$$
où $\tilde{J}^*$ est l'adjoint de $\tilde{J}$ lorsque espaces de
départ et d'arrivée sont munis du produit scalaire usuel.

Un instant de réflexion montre que, 
$$
\forall w\in \TT_z,\, \tilde{J}.\tilde{J}^*.w=J.J^*.w
$$

Si $w=J.s$ avec $s\in\SS_z$ et $w'=J.s'$ avec $s'\in\SS_z\cap(\kernel{J})^{\bot}$, alors, du fait que $\pdtscalaire{\tilde{J}^{-1}.J.s-s}{s'}=0$,
$$
\pdtscalaire{w}{w'}_z=\pdtscalaire{J.s}{J.s'}_z
\pdtscalaire{\tilde{J}^{-1}.J.s}{\tilde{J}^{-1}.J.s'}
=\pdtscalaire{\tilde{J}^{-1}.J.s}{s'}
=\pdtscalaire{s}{s'}
$$
De cette dernière formule, il vient que, en
posant
\begin{equation}
  \label{eq:def-grad}
  \nabla_z A = (-J_\alpha+C).a_z
\mathpunct{,}
\end{equation}
où $a_z$ est défini en \eqref{eq:def-az}, alors
$\nabla_z A \in \TT_z$ et on a, en reprenant \eqref{eq:def-AA} et \eqref{eq:def-az},
$$
\forall w\in\TT_z,\,\pdtscalaire{\nabla_z A}{w}_z = d_zA .w
$$

Ceci montre que $\nabla_z A$ est le gradient de l'aire restreinte à
la variété riemanienne $\PP_0(\ell)$.

On présente en \autoref{fig:gradient-exemple} le gradient calculé numériquement par la méthode du paragraphe \ref{sec:numerics}.

\begin{figure}[h!]
  \centering
  \includegraphics[scale=0.8]{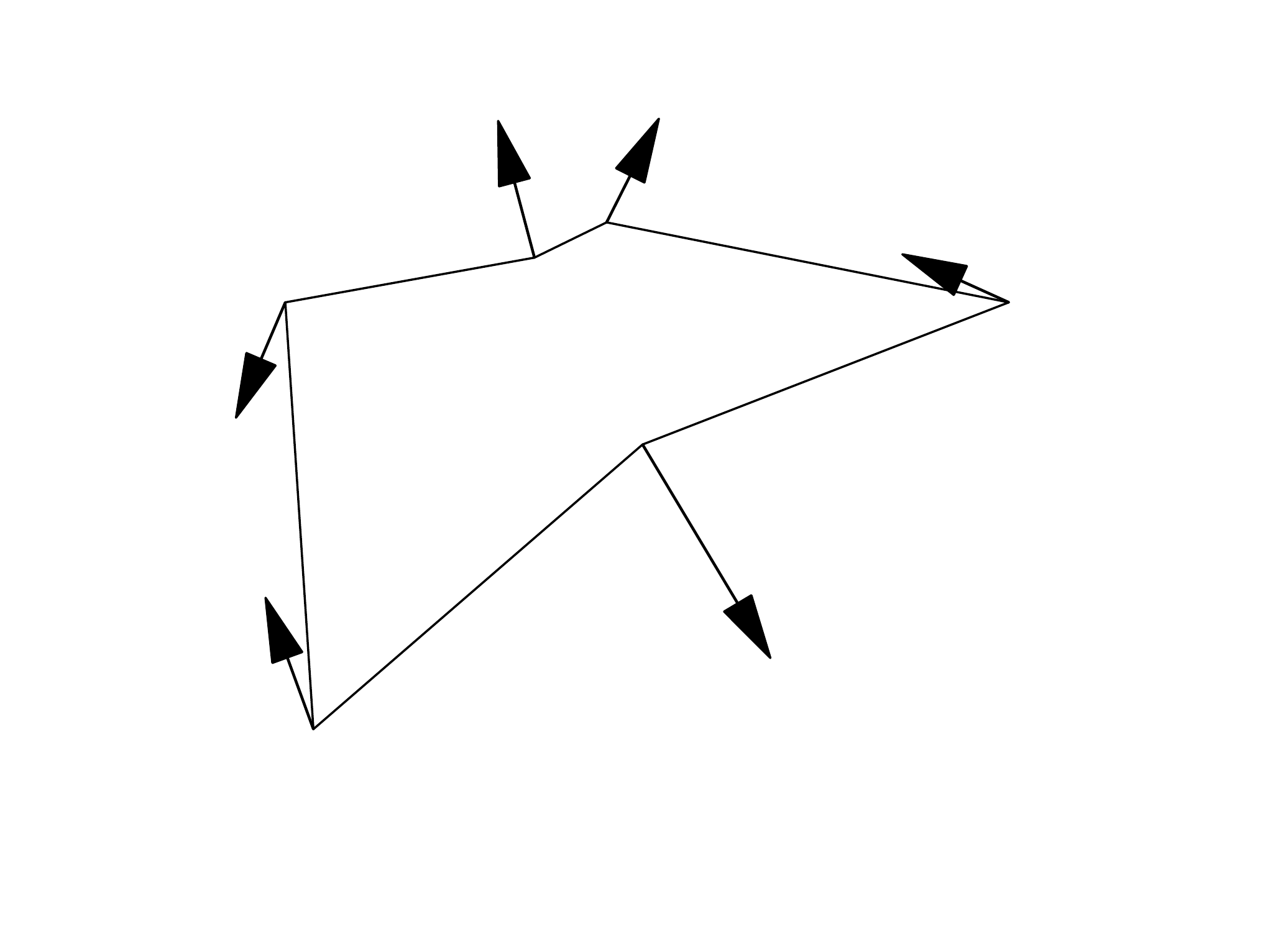}
  \caption{Un polygone à 6 côtés et le gradient de l'aire à côtés imposés}
  \label{fig:gradient-exemple}
\end{figure}

\subsection{Formules explicites pour $I$ et $K$}

Posons\footnote{La fonction $B$ ainsi définie est une variante du
  premier polynome de \ancetre{Bernoulli}. Les formules sont en analogie avec
  les faits suivants bien connus. Soit $E=\CC^\infty(\R) \cap \{f: f
  \text{ est $1$-périodique et  $\int_0^1 f(x)~dx=0$}\}$, $D:E\to E$,
  l'opérateur de dérivation est un isomophisme de réciproque $I:E\to
  E$ défini par
$$
\forall f\in E,\,\forall x\in\R,\, I(f)(x)=\int_0^1 B(y)f(x-y)~dy 
$$ 
où $B$ est la fonction $1$ périodique sur $\R$ définie par $B(0)=0$ et
$\forall t\in ]0,1[,\,B(t)=\frac12-t$.

} pour $t\in ]0,N[$ $B(t)=\frac{N-2t}{2N}$, $B(0)=0$ et étendons
$B$ à $\R$ par $N$ périodicité. $B$ est une fonction impaire. De la sorte, si $k\in \ZnZ$ ou $\kappa\in
\frac12+\ZnZ$, $B(k)$ et $B(\kappa)$ sont bien définis. 
Remarquons que si $t\in]\frac12,N-\frac12[$,
$$
B(t-\frac12)-B(t+\frac12)=\frac1N, B(t-\frac12)+B(t+\frac12)=B(t), 
$$
alors que pour $t=0$,
$$
B(t-\frac12)-B(t+\frac12)=\frac1N-1, B(t-\frac12)+B(t+\frac12)=B(t)=0, 
$$
De ceci, on déduit
\begin{lemma}\label{lemme:formules-I-K}
  Pour $z\in\C^S$, $\zeta\in\C^\Sigma$, $n\in S$, $\nu\in \Sigma$, on a
  \begin{eqnarray}
\label{eq:formules-I}(I.z)_\nu &=& \sum_{\kappa\in
  \frac12+\ZnZ}B(\kappa)z_{\nu-\kappa}\mathpunct{,} (I.\zeta)_n =\sum_{\kappa\in \frac12+\ZnZ}B(\kappa)\zeta_{n-\kappa}\mathpunct{.}\\
\label{eq:formules-K}(K.z)_n &=& \sum_{k\in \ZnZ}B(k)z_{n-k}\mathpunct{,} (K.\zeta)_\nu = \sum_{k\in \ZnZ}B(k)\zeta_{\nu-k}\mathpunct{.}
  \end{eqnarray}
\end{lemma}

\subsection{Numerics}\label{sec:numerics} Le numéricien qui sommeille en beaucoup d'entre
nous ne peut que tenter d'implémenter un algorithme de montée de
gradient.  Il s'agit de considèrer l'équation différentielle
\begin{equation}
  \label{eq:EDO-grad}
\frac{dz}{dt} = \nabla_z A  
\end{equation}

Comme $\PP_0(\ell)$ est une variété riemanienne compacte, pour toute
donnée initiale $z_0$, le problème de \ancetre{Cauchy} associé à
\eqref{eq:EDO-grad} et à la donnée initiale $z(0)=z_0$  admet une
unique solution $z$ sur $\R$.
Les limites de $z$ en $\pm\infty$ existent et sont des points critiques
de l'aire $A$ restreinte à $\PP_0(\ell)$.

Le long d'une trajectoire $z$, l'aire est croissante.

On peut implémenter rapidement un schéma d'\ancetre{Euler} 
visant à
maximiser l'aire. Une difficulté non négligeable est que la
discrétisation naïve ne conserve pas les longueurs imposées. 

Par convexité de chacun des ensembles $\{z\in\C_0^S, L^2_\nu(z)\leq
\ell^2_\nu\}$, à chaque étape, toutes les
longueurs $L_\nu$ augmentent. Le
travail numérique nécessite la prise en considération
de ce phénomène ; la solution adoptée pour la petite
simulation présentée en \autoref{fig:montee-gradient-6} et
\autoref{fig:montee-gradient-7} est de reprojeter à chaque étape sur
$\PP_0(\ell)$ via une méthode de \ancetre{Newton}.

Le schéma numérique se traduit donc ainsi: Etant donnés un pas de
temps $dt$, une tolérance $\epsilon$, partant d'un polygone
$Z_0\in\C_0^S$, $(\ell^2_\nu)_\nu=(L^2_\nu(Z_0))_\nu$, on construit par récurrence la
suite $(Z_k)_{k\in\N}$ par, à chaque pas $k\in\N$,
\begin{itemize}
\item une étape \ancetre{Euler}
$$
Z'_{k+1}=Z_k+dt.\nabla_{Z_k} A
$$
D'un point de vue très pratique, pour calculer
$\nabla_z A$, étant donné un polygone $z$, on calcule d'abord le
vecteur $\alpha$ défini par la formule \eqref{eq:def-alpha}, puis $\tilde{a_z}$ et $a_z$ en utilisant
les formules     \eqref{eq:formule-atildez} et
\eqref{eq:formule-az} et enfin $\nabla_z A=(-J_\alpha+C).a_z$ en
utilisant la définition \eqref{eq:def-Jalpha} de $J_\alpha$ ainsi que
la formule \eqref{eq:formules-I} pour $I$.
\item suivie d'une étape \ancetre{Newton}, \ie une variation
multidimensionnelle de la méthode de Babylone pour l'extraction de
racines carrées.
$$W_0=Z'_{k+1},\,\forall p\in \N,\, W_{p+1}=\frac12.\left(W_p
+I.\opdiag{\left(\frac{\ell^2_\nu}{L^2_\nu(W_p)}\right)_\nu}.D.W_p\right)$$
puis
$$
Z_{k+1}=W_p
$$
de sorte que $(\ell^2_\nu)_\nu$ et $(L^2_\nu(Z_{k+1}))_\nu$ soient $\epsilon$-proches.
\end{itemize}

\begin{figure}[h!]
  \centering
  \includegraphics[scale=0.65]{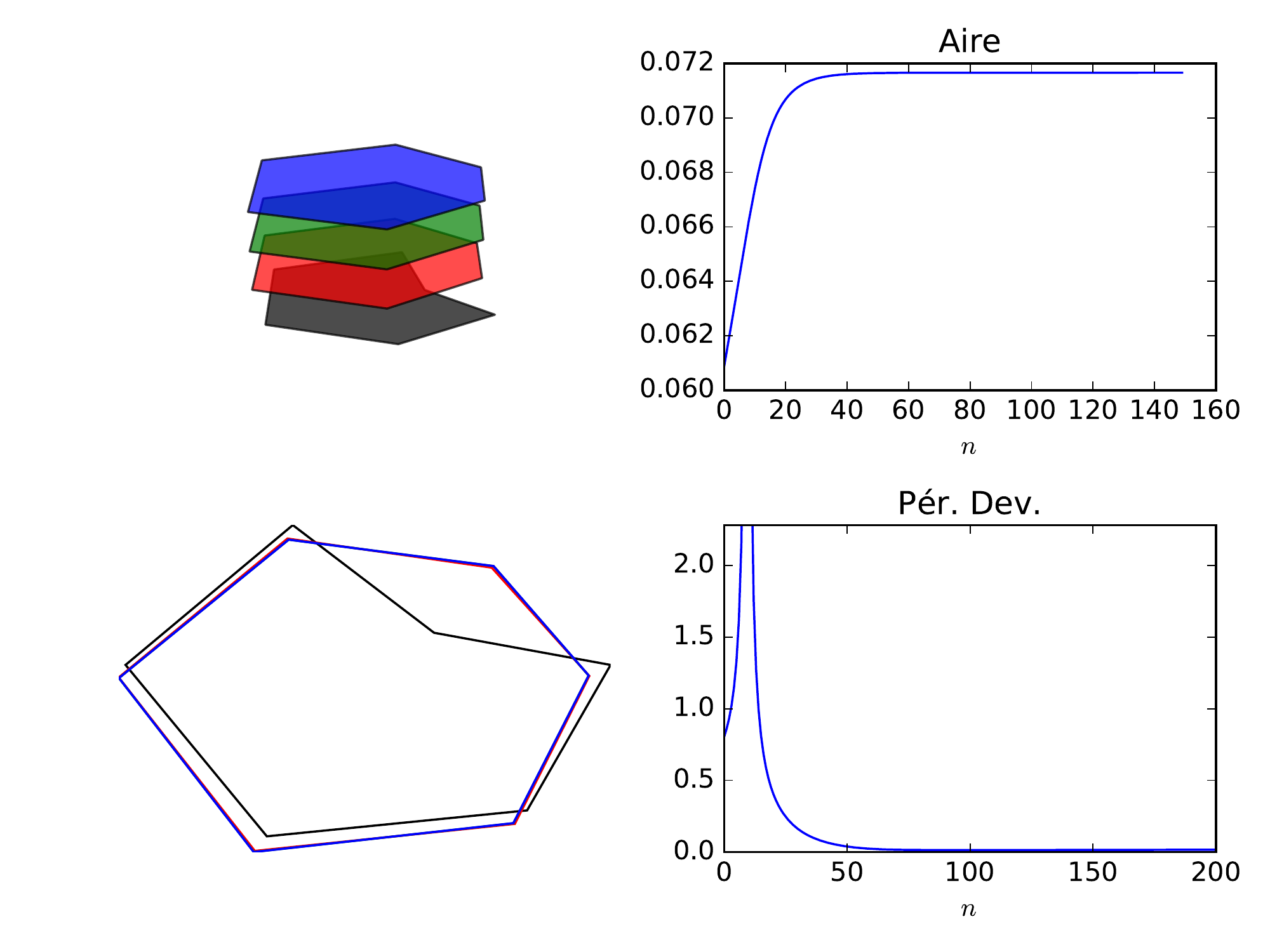}
  \caption{Evolution d'un polygone à 6 côtés, de son aire, du périmètre
    du polygone développé--c'est le polygone des centres des cercles
    circonscrits à 3 points consécutifs. On donne les vues dans le
    plan et, en 3D, dans des plans distincts, des
    polygones en début, $\frac13$, $\frac23$ et fin de l'évolution}
  \label{fig:montee-gradient-6}
\end{figure}

\begin{figure}[h!]
  \centering
  \includegraphics[scale=0.65]{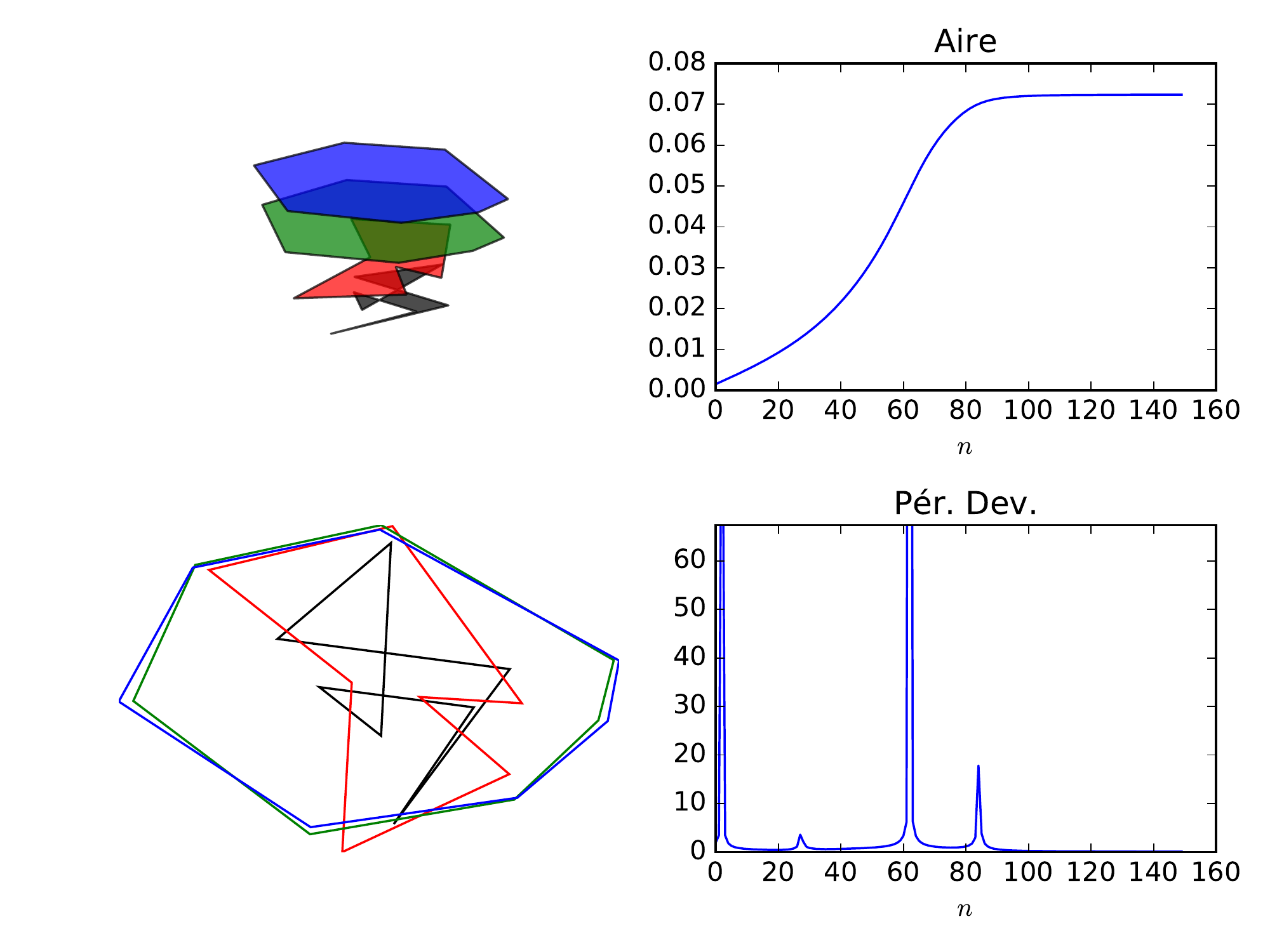}
  \caption{Evolution d'un polygone à 7 côtés, de son aire, du périmètre
    du polygone développé}
  \label{fig:montee-gradient-7}
\end{figure}

\section{Indices et dénombrement}\label{sec:applications}

\subsection{Indices}

L'expérience résumée en \autoref{fig:montee-gradient-7} soulève au
moins une question immédiate : l'aire augmente, comme prévu, mais le graphe du périmètre du polygone
développé, \ie le polygone $(o_n)_{n\in S}$ où $o_n$ est le centre du
cercle circonscrit aux points $z_{n-1},z_n$ et $z_{n+1}$ présente un
certain nombre de pics.

La longueur infinie de ce polygone marque le fait que l'un des $o_n$
est rejeté à l'infini, \ie la présence d'un triplet $z_{n-1},z_n$ et $z_{n+1}$ de
sommets consécutifs alignés, alors qu'une longueur nulle marque un
polygone $(z_n)$ cocyclique. 

Un pic marque  donc un phénomène de \og dépliement \fg au niveau d'un
sommet, le passage d'un angle au sommet de part et d'autre de l'angle
plat alors qu'un minimum marque l'approche d'un polygone cocyclique.

On en déduit que lors de l'évolution présentée en \autoref{fig:montee-gradient-7}, le polygone s'est approché, dans une direction d'aire croissante, d'un polygone cocyclique, critique, pour s'en éloigner ensuite.

Il s'agit d'une manifestation du comportement des trajectoires de
l'équation différentielle \eqref{eq:EDO-grad}  au voisinage d'un point
critique de $A$. Ce comportement est, au premier abord, déterminé par les
valeurs propres de la différentielle de $\nabla_z A$, autrement dit,
la signature de la Hessienne de $A$ restreinte à $\PP_0(\ell)$.

Ceci, entre autres considérations, pose la question du caractère non dégénéré\footnote{On remarque que, du fait de l'action diagonale de
$\S^1$ sur $\PP_0(\ell)$ et de l'invariance de $A$ sous cette action,
les points critiques de $A$ sont dégénérés au sens usuel. Cette
question est donc à considérer modulo cette
invariance, ce qui revient à travailler avec l'espace $\MM(\ell)$.}
 des points critiques de $A$ ainsi que celle de la détermination de
leur indice, \ie le nombre de valeurs
  propres négatives de la Hessienne. 

Ce travail a été mené dans le cas générique par \ancetre{Panina} et
\ancetre{Zhukova}  dans \cite{PZ2011} et \cite{Z2013}. Pour
schématiser le résultat principal de \cite{Z2013}, Théorème 3.1, on peut dire que, pour un $\Sigma$-uple
$\ell$ générique
\begin{enumerate}
\item Un point critique $z$ de $A$ est non dégénéré.
\item L'indice de \ancetre{Morse} $m(z)$, de $A$ en $z$, se calcule
  effectivement à l'aide de la géométrie de $z$. 
\end{enumerate}

\subsection{Dénombrement}

La théorie de \ancetre{Morse} permet d'estimer inférieurement le
nombre de points critiques d'une fonction\footnote{dont les points critiques
sont non dégénérés} sur une variété connaissant
la topologie de la variété. Les informations nécessaires--les nombres
de \ancetre{Betti} de $\MM(\ell)$-- sont calculés par \ancetre{Farber}
et \ancetre{Schütz} dans \cite{FS07}.

Le nombre de points critiques (et donc le nombre de configurations
cocycliques) dans $\MM(\ell)$ est minoré par la somme des nombres de
\ancetre{Betti} de $\MM(\ell)$, qui, d'après, \cite{FS07}, Theorem 2, est
inférieure à $2^{N-1} -\parmi{p}{N-1}$ où, suivant la parité, $N=2p+1$
ou $N=2p+2$. Cette somme dépend uniquement de propriétés combinatoires
du vecteur $\ell=(\ell_\nu)\in\R^{\Sigma}$ et peut se calculer dans
des cas particuliers. 

Un autre angle d'attaque pour tenter de dénombrer le nombre de
configurations cocycliques est de remarquer qu'il existe des
généralisations des formules de \ancetre{Heron} et de
\ancetre{Brahmagupta}, donnant respectivement l'aire d'un triangle et
d'un quadrilatère en fonction des longueurs des côtés.

Plus précisémment, \cf \cite{Pak2005}, \cite{MRR2005}, pour tout
$N\geq 3$, il existe un polynôme \footnote{Ce polynôme
  s'appelle le polynôme de \ancetre{Heron}--\ancetre{Robbins}. L'existence d'une telle
  relation polynomiale ne devrait pas être une surprise attendu que
  $A$ est valeur critique d'une fonctionnelle polynomiale sur une
  variété définie par des polynomes en les longueurs des côtés, il
  serait intéressant de voir si la théorie de l'élimination
  effectivement et informatiquement appliquée à l'obtention de la valeur de $A$ en  les points critiques du lagrangien $\LL$ donne des formules pour
  $\alpha_N$, par exemple pour $N=10$. Dans ce cas $\Delta_N=374$.}
homogène en $N+1$ variables, $\alpha_N$, à coefficients
entiers, unitaire en sa première variable, symétrique en les $N$ dernières variables, tel que, pour tout $N$-gone cocyclique, d'aire $A$, de
longueur de côtés $\ell=(\ell_\nu)$,
$$
\alpha_N(16A^2,(\ell_{\nu}^2)_{\nu\in\Sigma})=0
$$ 
D'après 
\cite{MRR2005}, $\alpha_N$ est, en $16A^2$, de degré
$
\Delta_N = \frac{N}{2}\parmi{p}{N-1}-2^{N-2}$ où, suivant la parité,
  $N=2p+1$ ou $N=2p+2$.

Dans les exemples $N=3$ et $N=4$, 
\begin{itemize}
\item 
La formule de \ancetre{Heron} donne l'aire d'un triangle $abc$ en fonction
des longueurs des côtés $\ell_\alpha,\ell_\beta,\ell_\gamma$
$$
A=\sqrt{p(p-\ell_\alpha)(p-\ell_\beta)(p-\ell_\gamma)} \text{ où } p=\frac12(\ell_\alpha+\ell_\beta+\ell_\gamma)
$$
On peut réécrire cette formule pour tenir aussi compte du cas du triangle
orienté négativement sous la forme d'une relation polynomiale
liant $A^2$, $\ell_\alpha^2$, $\ell_\beta^2$ et $\ell_\gamma^2$ :
$$
16A^2 
-2(\ell_\alpha^2\ell_\beta^2+\ell_\beta^2\ell_\gamma^2+\ell_\gamma^2\ell_\alpha^2)+\ell_\alpha^4+\ell_\beta^4+\ell_\gamma^4=0
$$
\item Pour un quadrilatère convexe cocyclique de longueurs des côtés $\ell_\alpha$, $\ell_\beta$, $\ell_\gamma$ et $\ell_\delta$, on dispose de la formule de \ancetre{Brahmagupta},

$$
A=\sqrt{(p-\ell_\alpha)(p-\ell_\beta)(p-\ell_\gamma)(p-\ell_\delta)} \text{ où }
p=\frac12(\ell_\alpha+\ell_\beta+\ell_\gamma+\ell_\delta)
\mathpunct{,}
$$
ce qui se réécrit
$$
\left(16A^2-(\ell_\alpha^2+\ell_\beta^2+\ell_\gamma^2+\ell_\delta^2)^2+2(\ell_\alpha^4+\ell_\beta^4+\ell_\gamma^4+\ell_\delta^4)\right)^2-64\ell_\alpha^2\ell_\beta^2\ell_\gamma^2\ell_\delta^2=0
$$
Cette formule couvre tous les cas possibles de quadrilatères dans
$\MM(\ell)$ en tenant compte de la symétrie et de la possibilité de
croiser le quadrilatère. 
\end{itemize}

Il existe par ailleurs, d'après \cite{MRR2005}, des exemples où un vecteur de longueurs de côté
$\ell=(\ell_\nu)$ donne lieu à $\Delta_N$ carrés d'aires distinctes de
$N$-gones cocycliques.
De notre point de vue, cela signifie que sur $\MM(\ell)$, $A$ prend 
$2\Delta_N$ valeurs distinctes et donc qu'il y a au moins $2\Delta_N$
configurations cocycliques distinctes dans $\MM(\ell)$.

\section{Complément~: Sous-variétés algébriques, calcul différentiel réel dans $\C^S$}\label{sec:calc-diff-rappels}

Une sous-variété algébrique réelle $\PP$ de $\C^S\simeq \R^{2N}$ est une partie de $\C^S\simeq\R^{2N}$ définie par un système de $p$ équations cartésiennes polynomiales réelles en $2N$ indéterminées. 

L'espace tangent, noté $\TT_z$, à $\PP$ en un point $z$ de $\PP$ est l'intersection des noyaux des différentielles en $z$ des $p$ polynômes définissants. Il s'agit d'un $\R$-sous-espace vectoriel de $\C^S$.

Pour être de plus une sous-variété \emph{différentielle} (de dimension $2N-p$) au voisinage d'un de ses points $z$, il suffit de demander que, par ailleurs, les différentielles en $z$ des $p$ polynomes définissants forment une famille libre, \ie que $\dim_\R \TT_z =  2N-p$.

Le théorème des fonctions implicites implique alors qu'au voisinage de $z$, quitte à renuméroter les axes, dans un parallélipède ouvert $P_z$, centré en $z$, $\PP$ est le graphe d'une fonction $f_z:\R^{2N-p}\to\R^p$ de classe $\CC^\infty$, \ie
$$
\PP\cap P_z = \{ (u,f_z(u)), u\in \R^{2N-p} \} \cap P_z
$$
et donne la paramétrisation de $\TT_z$ en graphe d'application $\R$-linéaire
$$
\TT_z =\{ (h,(d_z f_z).h),\, h\in \R^{2N-p} \}
$$

Un point $z$ de $\PP$ est dit \emph{régulier} s'il vérifie $\dim_{\R}\TT_z=2N-p$ , \emph{singulier} sinon.

Illustrons ces faits par un exemple très élémentaire.

Considérons dans $\C=\C^1\simeq \R^2$, et donc $N=1$, la partie
$$
\PP=\{z\in\C,\, \frac12\Im(z^2)=0\} \simeq \{(x,y)\in \R^2 ,\quad xy=0\}.
$$
Il s'agit de la réunion des deux axes de coordonnées de $\R^2$. C'est une sous-variété algébrique réelle de $\C$ contenant le point $z_0=0$  définie par un ($p=1$) seul polynôme définissant $P(z)=\frac12\Im(z^2)=xy$

On a 
$$
d_{z_0} P= y_0.dx+x_0.dy = 0 
$$
et donc
$$
\TT_{0}=\R^2,\, \dim_\R \TT_{0}=2 > 2N-p=1
$$
Le point $z_0=(x_0,y_0)$ est donc un point singulier de $\PP$ et le dessin montre qu'au moins en un sens naïf, la figure tangente à $\PP$ en ce point n'est certainement pas une droite.  

Pour tout $z\in\PP$, $z\not=z_0$, $d_z P$ est une forme linéaire non nulle sur $\R^2$ et donc $\dim_\R \TT_z =1 =2N-p$. Un tel $z$ est donc un point régulier de $\PP$.

Le calcul des différentielles peut-être mené directement en nombres complexes, sans avoir à repasser par les coordonnées réelles. Rappelons comment ce mécanisme fonctionne~:
Avec le choix $S=\{0,\dots,N-1\}$ de réalisation de $S$, l'isomorphisme naturel entre $\C^S$ et $\R^{2N}$ associe à tout $N$-uplet  $z=(z_n)_{n\in S}$ le $2N$-uplet de nombres réels $((x_n,y_n))_{n\in\{0,\dots,N-1\}}=(x_0,y_0,x_1,y_1,\dots)$ où, pour tout $n$, $z_n=x_n+iy_n$.

Si $f:\C^S\simeq\R^{2N}\to\C$ est une fonction différentiable, et $z\in\C^S$, la différentielle de $f$ en $z$ est l'application $\R$-linéaire, $d_z f:\C^S \to \C$ vérifiant
$$
d_z f = \sum_{n}\frac{\partial f}{\partial x_n}(z)d_z x_n + \sum_{n}\frac{\partial f}{\partial y_n}(z)d_z y_n
$$
où $d_z x_n : \C^S \to \C$, resp. $d_z y_n : \C^S \to \C$, sont définies par
$$
\forall t\in\C^S,\, (d_z x_n).t =\Re(t_n),\,
\text{ resp.}
(d_z y_n).t =\Im(t_n),\,
$$ 
Ces applications $\R$-linéaires ne dépendent\footnote{nous oublions donc l'indiçage par $z$} pas de $z$ et forment une base du $\R$-espace vectoriel $\LL(\C^S,\C)$.

En appliquant ce qui vient d'être dit aux fonctions \og coordonnées complexes\fg\  $z_n:\C^S\to\C$, resp. $\overline{z}_n$, associant à tout vecteur $z\in\C^S$, $z_n$ sa composante d'indice $n$, resp. le conjugué de celle-ci, on obtient les relations, en $z\in\C^S$ fixé,
$$
\forall n\in \C^S,\, 
d_z z_n = dx_n +i.dy_n
\et
d_z \overline{z}_n = dx_n -i.dy_n
$$
Les applications $\R$-linéaires $d_z z_n$ et $d_z \overline{z}_n$ ne dépendent pas de $z$ et forment une base du $\R$-espace vectoriel $\LL(\C^S,\C)$. On a
$$
\forall n\in\S,\, dx_n= \frac12dz_n+\frac12d\overline{z}_n \et
dy_n= \frac1{2i}dz_n-\frac1{2i}d\overline{z}_n \
$$

Si $f:\C^S\simeq\R^{2N}\to\C$ est une fonction différentiable, et $z\in\C^S$, la différentielle de $f$ en $z$ se réécrit donc
$$
d_z f = \sum_{n}\frac12\left(\frac{\partial f}{\partial x_n}(z)-i\frac{\partial f}{\partial y_n}(z)\right)d z_n 
+ \sum_{n}\frac12\left(\frac{\partial f}{\partial x_n}(z)+i\frac{\partial f}{\partial y_n}(z)\right)d\overline{z}_n
$$
On convient de noter 
$$
\frac{\partial f}{\partial z_n}=\frac12\left(\frac{\partial f}{\partial x_n}-i\frac{\partial f}{\partial y_n}\right)
\et
\frac{\partial f}{\partial \overline{z}_n}=\frac12\left(\frac{\partial f}{\partial x_n}+i\frac{\partial f}{\partial y_n}\right)
$$
De façon à obtenir
$$
d_z f = \sum_{n}\frac{\partial f}{\partial z_n}(z)d z_n 
+ \sum_{n}\frac{\partial f}{\partial \overline{z}_n}(z)d\overline{z}_n
$$

Lorsque $f$ est polynomiale en les variables $x_n,y_n$, on peut la réécrire en tant que polynôme en $z_n,\overline{z}_n$ et les polynômes $\frac{\partial f}{\partial z_n}$ et $\frac{\partial f}{\partial \overline{z}_n}$ se calculent en dérivant partiellement \og comme si \fg\ les variables $z_n,\overline{z}_n,n\in S$ étaient toutes indépendantes. 

A titre d'exemple, considérons la fonction $f:\C^2\to \R$ définie par
$$
\forall (z_1,z_2)\in\C^2,\, f(z_1,z_2)=|z_1|^2 + 2\Re(z_1.z_2) = z_1.\overline{z}_1+z_1.z_2+\overline{z}_1.\overline{z}_2
$$
On a alors
$$
\frac{\partial f}{\partial z_1}=\overline{z}_1+z_2,\,
\frac{\partial f}{\partial \overline{z}_1}=z_1+\overline{z}_2,\,
\frac{\partial f}{\partial z_2}=z_1,\,
\frac{\partial f}{\partial \overline{z}_2}=\overline{z}_2
$$
et donc 
$$
d_z f= (\overline{z}_1+z_2)dz_1+(z_1+\overline{z}_2)d\overline{z}_1
+(z_1)dz_2+(\overline{z}_1)d\overline{z}_2
=2.\Re((\overline{z}_1+z_2)dz_1+(z_1)dz_2)
$$

Dans le cas de notre exemple simple de sous-variété algébrique, on a
$$
\PP=\{z\in\C,\, \frac12\Im(z^2)=0\}
,
$$
on a, pour $z\in\C^1$, $P(z)=\frac12\Im(z^2)= -i\frac14z^2+i\frac14\overline{z}^2$ et
$$
\frac{\partial f}{\partial z}(z)=-\frac12iz,\,\frac{\partial f}{\partial \overline{z}}(z)=+\frac12i\overline{z}
$$
d'où
$$
d_z P = -\frac12iz.dz +\frac12i\overline{z}.d\overline{z}=\Im(z.dz)
$$
Comme $P$ est un polynôme réel, $d_z P$ est une forme linéaire sur le $\R$-ev $\C^1$ et on a,
$$
\TT_z=\{t\in\C^1,\, 2\Im(z.t)=0\}
$$
Pour $z\in \PP$ fixé, le $\R$-sous espace vectoriel $\TT_z$ de $\C^1$ est de dimension $1$ si $z\not=0$ (points réguliers de $\PP$), de dimension $2$ si $z=0$ (le point singulier).

\section{Quelques preuves de lemmes}\label{sec:preuves-lemmes}

\subsection{Démonstration du lemme \oldref{lemme:prop-elem-D-M}.}
\label{sec:preuve-lemme:prop-elem-D-M}
\begin{enumerate}
\item Montrons que $D=-D^*$. On considère dans cette égalité $D:\C^S\to \C^\Sigma$ définie par
$$
\forall z\in\C^S,\, \forall \nu\in\Sigma,\, (D.z)_{\nu} = z_{\nu+\demi}-z_{\nu-\demi}
$$
et donc $D^*:\C^\Sigma \to \C^S$ et
$D:\C^\Sigma\to \C^S$ définie par
$$
\forall \zeta\in\C^\Sigma,\, \forall n\in S,\, (D.\zeta)_{n} = \zeta_{n+\demi}-\zeta_{n-\demi}
$$
Les applications $D$, $D^*$ et $D$ sont $\C$-linéaires.

Soit $\zeta\in\C^\Sigma$, $D^*.\zeta$ est caractérisé dans $\C^S$ par le fait que
$$
\forall z\in\C^S,\, \pdthermitien{D^*.\zeta}{z}= \pdthermitien{\zeta}{D.z} 
$$
Soit $z\in\C^S$. On a
\begin{eqnarray*}
  \pdthermitien{D^*.\zeta}{z}&=& \pdthermitien{\zeta}{D.z}\\
&=& \sum_{\nu\in\Sigma} \zeta_\nu.\overline{z_{\nu+\demi}-z_{\nu-\demi}}\\
&=& \sum_{\nu\in\Sigma} \zeta_\nu.\overline{z}_{\nu+\demi}-\sum_{\nu\in\Sigma} \zeta_\nu.\overline{z}_{\nu-\demi}\\ 
&=& \sum_{n\in S} \zeta_{n-\demi}.\overline{z}_{n}-\sum_{n\in S} \zeta_{n+\demi}.\overline{z}_{n}
\mathpunct{.}
\end{eqnarray*}
On a fait deux changements d'indices de sommation en posant respectivement $n=\nu+\demi$ dans la première somme et  $n=\nu-\demi$ dans la deuxième somme. Finalement on a
\begin{eqnarray*}
  \pdthermitien{D^*.\zeta}{z}
&=& \sum_{n\in S} (\zeta_{n-\demi}-\zeta_{n+\demi}).\overline{z}_{n}
= -\pdthermitien{D.\zeta}{z}
\end{eqnarray*}
et donc $D^*=-D$.

Le même type de calculs donne que $M^*=M$, toujours avec ce jeu du typage basé sur la considération des espaces d'arrivée et de départ.

\item On a, par définition, pour $z\in\C^S$, 
$$\pi_0.z = z-\frac1N\pdthermitien{z}{\carac}\carac$$
et donc
$\pi_0.\carac=0$.  Comme $D.\carac=0$, il vient, pour $z\in\C^S$, 
$$
D.pi_0.z = D.z-\frac1N\pdthermitien{z}{\carac}D.\carac=D.z
$$
et, comme, pour $\zeta\in\C^\Sigma$,
$$\pi_0.\zeta = \zeta-\frac1N\pdthermitien{\zeta}{\carac}\carac$$
alors, en utilisant le fait que $D^*=-D$,
$$
pi_0.D.z = D.z-\frac1N\pdthermitien{D.z}{\carac}\carac
= D.z+\frac1N\pdthermitien{z}{D.\carac}\carac
=D.z
\mathpunct{.}
$$
En résumé 
$$
\forall z\in\C^S,\, \pi_0.D.z = D.\pi_0.z =D.z
\mathpunct{,}
$$
ce qui est ce que l'on cherche.
\item De la même manière, on a $M.\carac = \carac$ et donc, pour $z\in\C^S$,
$$
\pi_0.M.z =
M.z-\frac1N\pdthermitien{M.z}{\carac}\carac
= M.z-\frac1N\pdthermitien{z}{M.\carac}\carac
= M.z-\frac1N\pdthermitien{z}{\carac}\carac
\mathpunct{,}
$$
$$
M.\pi_0.z =
M.z-\frac1N\pdthermitien{z}{\carac}M.\carac
= M.z-\frac1N\pdthermitien{z}{\carac}\carac
$$
et donc
$$
\pi_0.M.z = M.\pi_0.z 
\mathpunct{,}
$$
ce qui est annoncé.
\item Soit $z\in\C^S$, on a, pour $\nu\in\Sigma$,
$$
(D.z)_\nu=z_{\nu+\demi}-z_{\nu-\demi}
\et
(M.z)_\nu=\frac12(z_{\nu+\demi}+z_{\nu-\demi})
$$
et donc, pour $n\in S$,
\begin{eqnarray*}
(M.D.z)_n
&=&\frac12\left(
(D.z)_{n+\demi}+(D.z)_{n-\demi}
\right)\\
&=&\frac12\left(
(z_{(n+\demi)+\demi}-z_{(n+\demi)-\demi})
+
(z_{(n-\demi)+\demi}-z_{(n-\demi)-\demi})
\right)
=\frac12\left(
z_{n+1}-z_{n-1}
\right)\\
\end{eqnarray*}
et
\begin{eqnarray*}
(D.M.z)_n
&=& (M.z)_{n+\demi}-(M.z)_{n-\demi}\\
&=&\frac12\left(
(z_{(n+\demi)+\demi}+z_{(n+\demi)-\demi})
-
(z_{(n-\demi)+\demi}+z_{(n-\demi)-\demi})
\right)
=\frac12\left(
z_{n+1}-z_{n-1}
\right)
\end{eqnarray*}
et donc $M.D.z=D.M.z$ avec a formule annoncée.
\item (Formule de \ancetre{Leibniz}). Supposons $\alpha\in\C^S$. On a, pour $z\in\C^S$, pour $n\in S$,
$$
(\opdiag{\alpha}.z)_n=\alpha_n.z_n
\mathpunct{,}
$$
et donc, pour $\nu\in\Sigma$, d'une part
$$
(M.\opdiag{\alpha}.z)_\nu=\frac12\left(\alpha_{\nu+\demi}.z_{\nu+\demi}+\alpha_{\nu-\demi}.z_{\nu-\demi}\right)
\mathpunct{,}
$$
et d'autre part
\begin{eqnarray*}
(\opdiag{M.\alpha}.D.z)_\nu&=&(M.\alpha)_\nu.(z_{\nu+\demi}-z_{\nu-\demi})  
=\frac12.(\alpha_{\nu+\demi}+\alpha_{\nu-\demi}).(z_{\nu+\demi}-z_{\nu-\demi})  \\
(\opdiag{D.\alpha}.M.z)_\nu&=&\frac12(D.\alpha)_\nu.(z_{\nu+\demi}+z_{\nu-\demi})  
=\frac12.(\alpha_{\nu+\demi}-\alpha_{\nu-\demi}).(z_{\nu+\demi}+z_{\nu-\demi})  \\
(\opdiag{M.\alpha}.D.z)_\nu+
(\opdiag{D.\alpha}.M.z)_\nu&=&
\frac12\left(\alpha_{\nu+\demi}.z_{\nu+\demi}+\alpha_{\nu-\demi}.z_{\nu-\demi}\right)
\end{eqnarray*}
et donc, pour $z\in\C^S$, $\nu\in\Sigma$,
$$
(\opdiag{M.\alpha}.D.z)_\nu+
(\opdiag{D.\alpha}.M.z)_\nu=
(M.\opdiag{\alpha}.z)_\nu
$$
\ie 
$$
\opdiag{M.\alpha}.D.z+
\opdiag{D.\alpha}.M.z=
M.\opdiag{\alpha}.z
$$
\ie 
$$
\opdiag{M.\alpha}.D+
\opdiag{D.\alpha}.M=
M.\opdiag{\alpha}
$$
\end{enumerate}

\subsection{Démonstration du lemme \oldref{lemme:prop-I-K}.}
\label{sec:preuve-lemme:prop-I-K}
\begin{enumerate}
\item
Soit $\zeta\in\C^\Sigma$, $I^*.\zeta$ est caractérisé dans $\C^S$ par le fait que
$$
\forall z\in\C^S,\, \pdthermitien{I^*.\zeta}{z}= \pdthermitien{\zeta}{I.z} 
$$
Soit $z\in\C^S$. On a
\begin{eqnarray*}
  \pdthermitien{I^*.\zeta}{z}&=& \pdthermitien{\zeta}{I.z}
\mathpunct{.}
\end{eqnarray*} 
Maintenant, $z=\pi_0.z+\frac1N\pdthermitien{z}{\carac}.\carac$ et, il existe (un unique) $\chi\in\C_0^\Sigma$ tel que $\pi_0.z=D.\chi$. On a donc 
$$
I.z = \chi
$$
De façon symétrique, il existe (un unique) $w\in\C_0^S$, $w=I.\zeta$, tel que $\pi_0.\zeta=D.w$ et donc
\begin{eqnarray*}
  \pdthermitien{I^*.\zeta}{z}&=& \pdthermitien{\zeta}{\chi}\\
&=& \pdthermitien{D.w}{\chi}+\frac1N\pdthermitien{\zeta}{\carac}\underbrace{\pdthermitien{\carac}{\chi}}_{=0}\\
&=& \pdthermitien{w}{D^*.\chi}
=- \pdthermitien{w}{D.\chi}\\
&=&
-\pdthermitien{I.\zeta}{\pi_0.z}
=-\pdthermitien{I.\zeta}{z}
\mathpunct{.}
\end{eqnarray*} 
On a donc $I^*=-I$. La relation $K^*=-K$ est due au fait que
$$
K^*=(M.I)^*=(I.M)^*=M^*.I^*=-M.I=-K^*
$$
\item Par définition, $I$ et $D$ sont inverses sur $\carac^\bot$ et s'annulent  $\vect{\carac}$. On a donc $I.D=D.I=\pi_0$ car $\pi_0$ est l'identité sur $\carac^\bot$ et s'annule sur le supplémentaire $\vect{\carac}$.
\item En composant par $I$ à droite la formule de \ancetre{Leibniz}, on obtient
  $$
D.\opdiag{\alpha}.I = \opdiag{M.\alpha}.D.I+\opdiag{D.\alpha}.MI
$$
et donc
$$
D.\opdiag{\alpha}.I = \opdiag{M.\alpha}.\pi_0+\opdiag{D.\alpha}.K
$$
\item En passant à l'adjoint cette dernière identité, il vient
$$
(-I).\opdiag{\overline{\alpha}}.(-D) = \pi_0.\opdiag{M.\overline{\alpha}}+(-K).\opdiag{D.\overline{\alpha}}
$$
et donc, quitte à substituer, $\alpha$ à $\overline{\alpha}$
$$
I.\opdiag{\alpha}.D = \pi_0.\opdiag{M.\alpha}-K.\opdiag{D.\alpha}
$$

\end{enumerate}
\subsection{Démonstration du lemme \oldref{lemme:resolution-systeme}.}
\label{sec:lemme:resolution-systeme}

  Soient $n$ et $p$ deux entiers naturels et soit $A$ une matrice à
  coefficients complexes\footnote{Pour une matrice, ou un vecteur $A$ à coefficients complexes, on note
$\overline{A}$ la matrice ou le vecteur dont les entrées sont les
conjuguées des entrées de $A$.}, $p$ lignes, $n$ colonnes.
On considère les ensembles suivants 
\begin{eqnarray*}
  \TT&=&\{T\in\C^n,\,A.T+\overline{A}.\overline{T}=0\},\\
  \TT_0&=&\{(T,S)\in\C^n\times\C^n,\,A.T+\overline{A}.S=0\},\\
  \TT_R&=&\{(T,S)\in\C^n\times\C^n,\,A.T+\overline{A}.S=0,T=\overline{S}\},\\
  \TT_I&=&\{(T,S)\in\C^n\times\C^n,\,A.T+\overline{A}.S=0,T=-\overline{S}\},
\end{eqnarray*}
alors 
\begin{enumerate}
\item Il est clair que $\TT_0$ est un $\C$-sev de $\C^n\times\C^n$, c'est le noyau de l'application $\C$-linéaire 
$$
\begin{array}{ccc}
 \C^n\times\C^n &\to & \C^p \\
(T,S)&\mapsto & A.T+\overline{A}.S
\end{array}
$$
Dans la même ligne de raisonnement, les applications
$$
\begin{array}{ccc}
 \C^n\times\C^n &\to & \C^n \\
(T,S)&\mapsto & T-\overline{S}
\end{array}
\et
\begin{array}{ccc}
 \C^n\times\C^n &\to & \C^n \\
(T,S)&\mapsto & T+\overline{S}
\end{array}
$$
soint $\R$-linéaires (la conjugaison ne permet pas la $\C$-linéarité). Les ensembles $\TT_R$, resp. $\TT_I$, est intersection du noyau de la première, resp. de la seconde, et du $\R$-espace $\TT_0$.

Finalement,  $\TT$ est un $\R$-sev de
  $\C^n$ en tant que noyau de l'application $\R$ linéaire
$$
\begin{array}{ccc}
 \C^n &\to & \C^p \\
(T,S)&\mapsto & A.T+\overline{A}.\overline{T}
\end{array}
$$

\item $\TT_R$ et $\TT_I$ sont des $\R$-sous-espaces de $\TT_0$.
Définissons les applications $\R$-linéaires $\pi_{R//I},\pi_{I//R}:\TT_0\to \C^n\times\C^n$ par
$$
\forall (T,S)\in \TT_0\subset\C^n\times\C^n,\,
\pi_{R//I}(T,S)=\frac12(T+\overline{S},\overline{T}+S),\,
\pi_{I//R}(T,S)=\frac12(T-\overline{S},-\overline{T}+S).$$
En ajoutant et retranchant les équations, valables pour $(T,S)\in\TT_0$,
$$
A.T+\overline{A}.S=0 \et 
\overline{A}.\overline{T}+A.\overline{S}=0 
$$ 
On obtient que 
$$
\forall (T,S)\in \TT_0,\,
\pi_{R//I}(T,S)\in \TT_R \et
\pi_{I//R}(T,S)\in \TT_I
$$
et par ailleurs, on a clairement $\TT_R \cap \TT_I=\{0\}$ et 
$$
\forall (T,S)\in \TT_0,\,
(T,S)=\pi_{R//I}(T,S)+\pi_{I//R}(T,S)
$$
Donc $\TT_0 = \TT_R \oplus \TT_I$.

\item L'application $\TT_R\to\TT_I$, $(T,S)\mapsto (i.T,i.S)$ est clairement bien définie, $\R$-linéaire, bijective~: c'est   un isomorphisme de $\R$-espaces vectoriels.

L'application $\TT_R\to \TT$, $(T,S)\mapsto T$ est clairement bien définie, $\R$-linéaire, bijective~: c'est   un isomorphisme de $\R$-espaces vectoriels.
\item Les $\R$-espaces vectoriels $\TT$, $\TT_R$ et $\TT_I$ sont tous de même dimension (réelle) et, comme $\TT_0 = \TT_R \oplus \TT_I$, $\dim_\R \TT_0=\dim_\R \TT_R+\dim_\R \TT_I=2\dim_\R \TT$. Du fait que $\TT_0$ est aussi un $\C$-espace vectoriel, $\dim_\R \TT_0=2\dim_\C \TT_0$, ce qui entraine les égalité annoncées.
\item Si $J$ et $Q$ sont deux matrices telles que
$$\TT_0=\{(T,S)\in\C^n\times\C^n,T=J.S,Q.S=0\}$$
alors
\begin{eqnarray*}
  \TT_R&=&\pi_{R//I}(\TT_0)\\
&=& \{(T+\overline{S},\overline{T}+S)\in\C^n\times\C^n,\,\exists (T,S)\in\C^n\times \C^n,\,T=J.S,Q.S=0\}\\
&=& \{(J.S+\overline{S},\overline{J.S}+S)\in\C^n\times\C^n,\,\exists S\in\C^n,\,Q.S=0\}
\end{eqnarray*}
et, via la bijection entre $\TT$ et $\TT_R$, 
$$\TT=\{J.S+\overline{S},\,\exists S\in\C^n, Q.S=0\}=(J+C)(\kernel{Q})$$
Les égalités de dimensions précédentes donnent que
$$\dim_\R \TT=\dim_\R \TT_R=\dim_\C \TT_0$$
L'application $\C$-linéaire
$\left\{
\begin{array}{ccc}
 \kernel{Q} &\to & \C^n\times \C^n \\
S&\mapsto & (J.S,S)
\end{array}\right.
$
est injective, d'image $\TT_0$, on a donc 
$\dim_\C \TT_0= \dim_\C\kernel{Q}$
et finalement
$$\dim_\R \TT=\dim_\C\kernel{Q}
\mathpunct{.}
$$ 
\end{enumerate}

\bibliographystyle{halpha} 
\bibliography{Polygones}

\end{document}